\documentclass[sn-mathphys-num]{sn-jnl}

\usepackage{graphicx}%
\usepackage{multirow}%
\usepackage{amsmath,amssymb,amsfonts}%
\usepackage{amsthm}%
\usepackage{mathrsfs}%
\usepackage[title]{appendix}%
\usepackage{xcolor}%
\usepackage{textcomp}%
\usepackage{manyfoot}%
\usepackage{booktabs}%
\usepackage{algorithm}%
\usepackage{algorithmicx}%
\usepackage{algpseudocode}%
\usepackage{listings}%

\theoremstyle{thmstyleone}%
\newtheorem{theorem}{Theorem}
\newtheorem{prop}[theorem]{Proposition}%

\theoremstyle{thmstyletwo}%
\newtheorem{oss}{Remark}%

\theoremstyle{thmstylethree}%
\newtheorem{definition}{Definition}%
\newtheorem{lemma}{Lemma}

\usepackage{graphicx} 
\usepackage[utf8]{inputenc}
\usepackage[T1]{fontenc} 
\usepackage[british]{babel} 
\usepackage{amssymb} 
\usepackage{amsmath}
\usepackage{amsthm}
\usepackage{amsfonts} 
\usepackage{dsfont}
\usepackage{amscd}
\usepackage{bm}
\usepackage{bbm}
\usepackage{color}
\usepackage{wasysym} %
\usepackage{mathtools}
\usepackage{mathrsfs}
\usepackage[makeroom]{cancel}
\usepackage{version}
\usepackage{xcolor}
\usepackage{textcomp}
\usepackage[mathscr]{euscript}
\usepackage{epsfig}
\usepackage{array}
\usepackage{indentfirst}
\usepackage{float}
\usepackage{color}
\usepackage{xcolor}
\usepackage{algorithm}
\usepackage{enumerate}
\usepackage{caption}
\usepackage{subcaption} 
\usepackage{booktabs}
\mathtoolsset{showonlyrefs}
\usepackage{hyperref}
\usepackage{emptypage} 
\usepackage{afterpage}

\newcommand{\CX}{\mathcal{X}}

\newcommand{\R}{\mathbb{R}}
\newcommand{\N}{\mathbb{N}}

\newcommand{\diff}{\mathrm{d}}
\newcommand{\sign}{\operatornamewithlimits{sign}}
\newcommand{\argmin}{\operatornamewithlimits{argmin}}
\newcommand{\argmax}{\operatornamewithlimits{argmax}}

\def\Mx{{\mathcal{M}(\Omega)}}
\def\Mxpos{{\mathcal{M}^+(\Omega)}}
\def\kl{\mathcal{D}_{KL}}	
\def\dkl{\tilde{\mathcal{D}}_{KL}}	

\graphicspath{{Figures}}
\newcommand{\red}[1]{{#1}}

\begin{document}

\title[Off-the-grid regularisation for Poisson inverse problems]{Off-the-grid regularisation for Poisson inverse problems}


\author*[1,2]{\fnm{Marta} \sur{Lazzaretti}}\email{lazzaretti@dima.unige.it}

\author[1]{\fnm{Claudio} \sur{Estatico}}\email{estatico@dima.unige.it}

\author[4]{\fnm{Alejandro} \sur{Melero}}\email{
alejandro.melero@achucarro.org}

\author[2,3]{\fnm{Luca} \sur{Calatroni}}\email{luca.calatroni@unige.it}


\affil[1]{\orgdiv{Dipartimento di Matematica}, \orgname{Università di Genova}, \orgaddress{\street{Via Dodecaneso 35}, \city{Genova}, \postcode{16145}, \country{Italy}}}

\affil[2]{\orgdiv{I3S Laboratory}, \orgname{CNRS-UniCA-Inria}, \orgaddress{\street{2000 Route des Lucioles}, \city{Sophia Antipolis}, \postcode{06903}, \country{France}}}

\affil[3]{\orgdiv{Machine Learning Genoa Center (MaLGa)}, \orgname{Università di Genova}, \orgaddress{\street{Via Dodecaneso 35}, \city{Genova}, \postcode{16145}, \country{Italy}}}

\affil[4]{\orgdiv{Laboratory of GTPases and Neurosignalling}, \orgname{Achucarro Basque Center for Neuroscience}, \orgaddress{\street{Parque Científico UPV/EHU, Edif. Sede, Planta 3}, \city{Leioa}, \postcode{E-48940}, \country{Spain}}}


\abstract{
 Off-the-grid regularisation has been extensively employed over the last decade in the context of ill-posed inverse problems formulated in the continuous setting of the space of Radon measures $\Mx$. These approaches enjoy convexity and counteract the discretisation biases as well the numerical instabilities typical of their discrete counterparts.
In the framework of sparse reconstruction of discrete point measures (sum of weighted Diracs), a Total Variation regularisation norm in $\Mx$ is typically combined with an $L^2$ data term modelling additive Gaussian noise. 
To assess the framework of off-the-grid regularisation in the presence of signal-dependent Poisson noise, we consider in this work a variational model where Total Variation regularisation is coupled with a Kullback-Leibler data term under a non-negativity constraint. Analytically, we study the optimality conditions of the composite functional and analyse its dual problem. Then, we consider an homotopy strategy to select an optimal regularisation parameter and use it within a Sliding Frank-Wolfe algorithm. Several numerical experiments on both 1D/2D/3D simulated and real 3D fluorescent microscopy data are reported.}

\keywords{Off-the-grid sparse regularisation, Poisson noise, Sliding Frank-Wolfe, fluorescence microscopy imaging.}

\maketitle

\section{Introduction}

Discrete (or \emph{on-the-grid}) sparse optimisation approaches are nowadays established techniques in the field of mathematical signal and image inverse problems.  In the context of linear ill-posed inverse problems, for instance, they aim at retrieving a sparse approximation of a quantity of interest $x$  (e.g., an image) from blurred, noisy and potentially incomplete acquisitions $y\in\R^M$ on a regular grid of size $N\geq M$, see~Fig.~\ref{fig:grid1}. The discretisation parameter $N$ determines the localisation precision, as only the centres of the discretisation intervals are candidates for reconstruction. As such, in case of rough discretisations,  very approximate reconstructions may be computed (see, e.g., Fig.~\ref{fig:grid2} where small values $M=N$ are used). To obtain higher precision, one typically chooses a grid-size parameter  $N>M$ for reconstruction, as in Fig.~\ref{fig:grid3}. 
Choosing a large value for $N$, however, may cause instabilities in the reconstructions \cite{Duval2017thingrids} due to the higher numerical complexity.

\textit{Off-the-grid} approaches aim at overcoming such difficulties. They can be thought indeed as the natural framework to deal with the case $N\to+\infty$ of \textit{on-the-grid} formulations \cite{Duval2017thingrids_b,Laville2021}.  In such framework, the spatial domain $\Omega\subseteq\R^d$ is not discretised by a regular grid, but, rather, the quantity of interest is modelled as an element of a suitable functional space defined on $\Omega$. A natural framework for spike reconstruction problems, is, for instance, the space of Radon measures $\mu\in\Mx$ where the quantity of interest can be modelled as the measure $\mu_{a,x}= \sum_{i=1}^n a_i \delta_{x_i}$  with $n$ being the number of spikes and where the question is therefore how to estimate the number of spikes $n$, and then how to retrieve both continuous positions $x_i\in\Omega$ and amplitudes $a_i\geq 0$ for $i=1,\ldots,n$.

Inverse problems in the space of measures and off-the-grid optimisation methods have been first proposed in \cite{Scherzer2009, bredies2013inverse,decastro2012exact,FernandezGranda2013} and since then they have been a topic of intense research activity for the mathematical community, both from an analytical and numerical viewpoint, see, e.g., \cite{duval2014exact,denoyelle2017support,boyer2017adapting,poon2019multi,Denoyelle2019,Flinth,Traonmilin2024}. Off-the-grid methods have been proved to be particularly useful in applications where fine-scale details need to be retrieved from noisy acquisitions, such as spike detection in astronomy and microscopy \cite{Denoyelle2019}, as well as parameter estimation in spectroscopy \cite{Dossal2017} and density mixture estimation \cite{DeCastro2021}. 
Standard approaches in this setting usually combine off-the-grid regularisation with an additive  (Gaussian) noise modelling on the underlying signal, which is in general easier to work with, from both an analytical and a computational point of view. Such modelling corresponds to the well-studied variational formulation of the Beurling LASSO (BLASSO) model \cite{bredies2013inverse,FernandezGranda2013}. Beyond Gaussian noise models, we also mention the recent work \cite{Pouchol2024}, where a study on the singularity of minimisers for general divergences defined on $\mathcal{M}(\Omega)$ under non-negativity constraints and no further explicit regularisation is carried out, along with numerical validations on exemplar medical imaging problems.

\medskip

In this work, we consider an off-the-grid modelling in $\Mx$ under the specific modelling assumption of signal-dependent Poisson noise in the data. This choice is motivated by some particular biological applications of interest, such as fluorescence microscopy, where, due to the photon emission nature of the light, Poisson noise is better suited than the Gaussian one to describe the process of photon counts on acquired images \cite{Bertero2018}. While in a discrete setting, a precise modelling of noise statistics if often not necessary due to the inevitable biases introduced by the choice of the regularisation employed, a natural question  is whether whenever a refined off-the-grid regularisation is used, a precise noise model could indeed be relevant.
As a graphical visualisation, we report in Figure \ref{fig:grid} a visual comparison between reconstructions obtained with on-the-grid approaches and the off-the-grid approach proposed in this work to solve a spike-deconvolution problem under the choice of the Poisson data term considered in this work.

\begin{figure}
    \centering
    \begin{subfigure}{0.4\textwidth}
        \includegraphics[width=\textwidth]{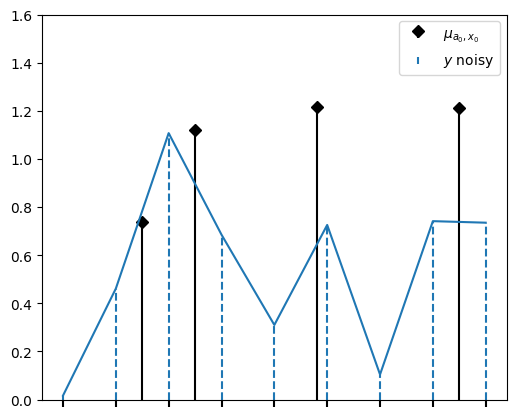}
        \caption{Acquisition $y\in\R^M$}
        \label{fig:grid1}
    \end{subfigure}
    \begin{subfigure}{0.4\textwidth}
        \includegraphics[width=\textwidth]{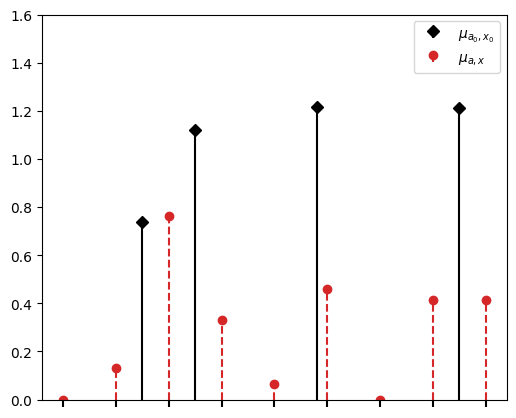}
        \caption{Discrete reconstruction $x\in\R^M$}
        \label{fig:grid2}
    \end{subfigure}\\
    \begin{subfigure}{0.4\textwidth}
        \includegraphics[width=\textwidth]{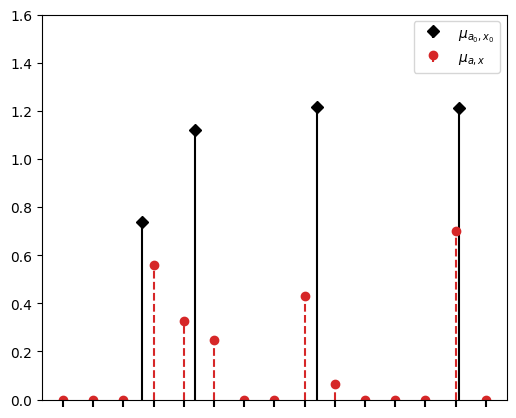}
        \caption{Discrete reconstruction $x\in\R^N$}
        \label{fig:grid3}
    \end{subfigure}
    \begin{subfigure}{0.4\textwidth}
        \includegraphics[width=\textwidth]{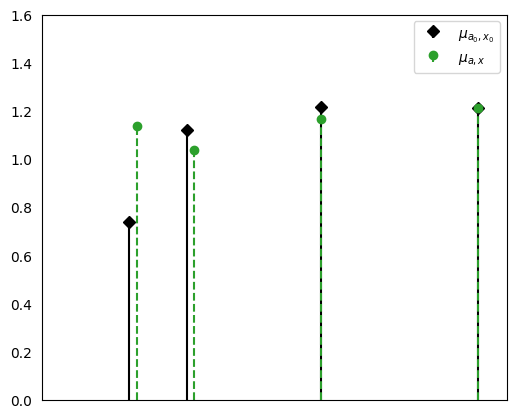}
        \caption{Off-the-grid rec. $\mu_{a,x}\in\Mx$}
        \label{fig:grid4}
    \end{subfigure}
    \caption{Comparison between discrete (on-the-grid) and off-the-grid sparse reconstructions with a Poisson data term. In black: the ground-truth spikes to retrieve. In Fig.\ref{fig:grid1}: the acquired blurred and noisy signal $y\in\R^M$ lying on a low-resolution grid of size $M$. In Fig.\ref{fig:grid2}: in red, a discrete reconstruction with support on a grid with $M$ pixels. In Fig.\ref{fig:grid3}: in red, discrete reconstruction with support on a grid with $N>M$ pixels. In Fig.\ref{fig:grid4}: 
    in green, off-the-grid reconstruction. 
    }
    \label{fig:grid}
\end{figure}

\subsection{Inverse problems in the space of Radon measures $\Mx$}

In this section we recall some important definitions and properties of the space of Radon measures $\Mx$ following \cite{poon2019sparse,Denoyelle2019,Laville2021}. For more details on Radon measures see \cite{Rudin1987,cohn2013measure}.

Let $\Omega\subseteq \R^d$, with $d\in\N$, $d\ge 1$, be a compact subset of $\R^d$ with non-empty interior. We denote by $\mathscr{C}(\Omega,\R)$ the space of real continuous functions $\psi: \Omega \rightarrow \R$ on $\Omega$.
The space of Radon measures can be defined through duality, see \cite{bredies2013inverse,Denoyelle2019,denoyelle2018_phdthesis,JAVANSHIRI2013}.
\begin{definition} 
\label{def_mx}
    The Banach space $\Mx$ of real signed Radon measures on $\Omega$ is the topological dual of $\mathscr{C}(\Omega,\R)$ endowed with the supremum norm $\|\cdot\|_{\infty,\Omega}$, defined by $\|\psi\|_{\infty,\Omega}:=\sup_{x\in\Omega}|\psi(x)|$.
\end{definition}

This definition allows to characterise any measure $\mu\in\Mx$ as a continuous linear form evaluated on continuous functions $\psi\in \mathscr{C}(\Omega,\R)$ in terms of the duality pairing, so that
    \begin{equation}
        (\forall\mu\in\Mx), \qquad \langle \psi,\mu\rangle_{\mathscr{C}(\Omega,\R)}\times\Mx=\int_\Omega \psi \diff \mu.
        \label{eq_prod_dual_misura}
    \end{equation}
A Radon measure $\mu\in\Mx$ is a positive measure if $\langle\psi,\mu\rangle_{\mathscr{C}(\Omega,\R)}\times\Mx$ is non-negative for any non-negative function $\psi\in \mathscr{C}(\Omega,\R)$. This specifies the meaning of the term \textit{signed} in the above definition, as the quantity $\langle\psi,\mu\rangle_{\mathscr{C}(\Omega,\R)}\times\Mx$ can be also negative.  

It can be shown that $\Mx$ is a non-reflexive Banach space endowed with the Total Variation (TV) norm, here defined for all $\mu\in\Mx$ as:
\begin{equation} \label{eq:TV_measure}
    |\mu|(\Omega)=\sup \Big( \int_\Omega \psi \diff \mu ~\Big| ~\ \psi\in\mathscr{C}(\Omega,\R),\ \|\psi\|_{\infty,\Omega}\le 1\Big).
\end{equation}
The TV norm is convex and lower semi-continuous with respect to the weak* topology, hence its subdifferential is nonempty and defined as \cite{bredies2013inverse,duval2014exact}
\begin{equation}
    \partial |\mu|(\Omega)=\left\{\psi\in\mathscr{C}(\Omega,\R)|\ \|\psi\|_{\infty,\Omega}\le 1 \text{ and } \int_\Omega \psi\diff \mu=|\mu|(\Omega)\right\}.
    \label{eq_subdiff_TV}
\end{equation}
Note that for sparse discrete measures, i.e. weighted sums of Diracs
    \begin{equation}
        \mu_{a, x} = \sum_{i=1}^N a_i \delta_{x_i}  \text{ with } N \in \mathbb{N},\ a=(a_1,\dots,a_N) \in \R^N,\ x=(x_1,\dots,x_N) \in \Omega^N,
        \label{eq_discrete_measures}
    \end{equation}
    with $x_i\not=x_j$ for $i\not=j$, 
the TV norm coincides with the $L^1$ norm of the amplitudes vector $a$, that is $|\mu_{a,x}|(\Omega)=\|a\|_1$.
This explains why the TV norm is considered in this setting as a generalisation of the $L^1$ norm. Moreover, in this special case one has $\,\langle \psi,\mu\rangle=\sum_{i=1}^N a_i \psi(x_i)\,$ and the subdifferential has the following expression involving the sign function
\begin{equation}
    \partial |\mu_{a,x}|(\Omega)=\left\{\psi\in\mathscr{C}(\Omega,\R)|\ \|\psi\|_{\infty,\Omega}\le 1,\ \forall i=1,\ldots,N\ \psi(x_i)=\sign(a_i)\right\}.
    \label{eq_subdiff_TV_discrete}
\end{equation}

We focus now on the formulation of linear inverse problems in $\Mx$. As an acquisition space we will consider a \red{functional} Hilbert space $\mathcal{H}$. 
Let $\mu \in \Mx$ be the unknown source measure. We consider an acquisition $\bar{y} \in \mathcal{H}$ being the result of the action of the forward operator $\Phi: \Mx \rightarrow\mathcal{H}$ evaluated on $\mu$, i.e. $\bar{y}=\Phi\mu$. The forward operator $\Phi: \Mx \rightarrow\mathcal{H}$ is defined in terms of a a continuous and bounded measurement kernel $\varphi: \Omega \rightarrow \mathcal{H}$, that is: 
\begin{equation}
    \Phi \mu:=\int_\Omega \varphi(x) \diff \mu(x).
    \label{eq_IP_offgrid}
\end{equation}
Note that this integral should not be confused with the concept of duality pairing \eqref{eq_prod_dual_misura}, which is a scalar function defined as the integral over $\Omega$ of a continuous real function with respect to a measure $\mu\in\Mx$. In \eqref{eq_IP_offgrid}, the integral is then a Bochner integral \cite{cohn2013measure} for vector-valued functions. Indeed, $\varphi$ is not a continuous real function in $\mathscr{C}(\Omega,\R)$, but rather $\varphi:x\in\Omega\mapsto \varphi(x)=\varphi_x(\cdot)\in\mathcal{H}$ is a map from $\Omega$ to $\mathcal{H}$. This implies that $\varphi(x)$ is not a real value but an element of $\mathcal{H}$. The integral is well-defined (as a Bochner integral) if $\varphi$ is continuous and bounded \cite{Chizat2018OnTG,Denoyelle2019}.

In the following, \red{we consider $\Mx$ with the weak* topology and its dual space $\mathscr{C}(\Omega,\R)$. In this setting, the forward operator $\Phi$ is weak*-weak continuous \cite{Denoyelle2019}.} Thus, it is possible to define the adjoint operator of $\Phi$ in the weak* topology, namely the map $\Phi^*:\mathcal{H}\longrightarrow \mathscr{C}(\Omega,\R)$ such that, for all $p\in\mathcal{H}$, $\Phi^*p$ is the real continuous function defined by
\begin{equation}
    \Phi^*p(x):=\Big(x\in\Omega\mapsto\langle p,\varphi_x(\cdot)\rangle_\mathcal{H}=\int_\Omega p(t)\varphi_x(t)\diff t\Big)
    \label{eq_phi_adjoint}
\end{equation}
for all $x\in\Omega$.

The choice of the kernel $\varphi$ and of the acquisition space $\mathcal{H}$ depends on the specific physical acquisition process. In the following, we consider a convolution kernel, which is of practical interest in fluorescence microscopy (see, e.g., \cite{Denoyelle2019,Koulouri_2021_AdaptiveSuperRes}). In this setting, a natural choice is therefore $\mathcal{H}=L^2(\Omega)$, with the convolution kernel $\varphi:\Omega\rightarrow L^2(\Omega)$ defined in terms of a Point Spread Function (PSF) $\tilde{\varphi}:\Omega\rightarrow \R$ acting as:
    \begin{equation}
        \varphi_x(s) := \tilde{\varphi}(s-x)\in\R  \ \forall x,s\in{\Omega}.
        \label{eq_conv_op}
    \end{equation}
Note that, depending on the microscopy technique used, one can have different PSFs. For instance, the Gaussian PSF, centred in $c \in \Omega$ with radius $\sigma>0$, is defined by 
    \begin{equation}
        s \mapsto \tilde{\varphi}(s-c) := {(2 \pi \sigma^2)}^{-d / 2} e^{-\|s-c\|_2^2 / 2 \sigma^2}.
    \end{equation}
For other possible choices of measurement kernels, we refer the reader to \cite{Denoyelle2019}.

We observe that the action of the forward operator $\Phi$ on finite linear combination of Diracs \eqref{eq_discrete_measures} can be explicited as follows: 
\begin{equation}
    \Phi \mu_{a, x} = \int_\Omega \varphi(x)\diff \mu(x)=\sum_{i=1}^N a_i\varphi(x_i).
    \label{eq_conv_operator}
\end{equation}
For simplicity, the following notation will therefore be used in the sequel $\Phi_x(a)=\sum_{i=1}^N a_i\varphi(x_i)$ to denote $\Phi \mu_{a, x}$.

\smallskip

In fluorescence microscopy imaging applications, the objects of interest are usually images of molecules, i.e. single point-sources emitting fluorescent light. The unknown is therefore well-described by  non-negative discrete measures of the form \eqref{eq_discrete_measures}.
Other possible objects of interest in this field are microtubules, 1-dimensional curve structures \cite{LavilleSIAMIS2023,LavilleSSVM2023}, and cells (2-dimensional), that can be modelled as piece-wise constant functions \cite{DeCastro2023,DeCastroPetit2021,DeCastro2024exact}. For simplicity, we will focus in the following only on inverse problems aiming at recovering (0-dimensional) measures of the form \eqref{eq_discrete_measures}. The extension to more general regularisation models is left for future work. 

\subsection{The BLASSO problem: formulation, duality and optimality conditions}
 \label{sec_blasso}
 
The standard sparse spike deconvolution problem, consists in recovering a (small) finite linear combination of Diracs $\mu_{a, x} = \sum_{i=1}^N a_i \delta_{x_i}$ from a blurred and noisy acquisition
\begin{equation} 
    y=\Phi \mu_{a, x}+\omega,
    \label{eq_ip_offthegrid_gaussian}
\end{equation}
where $\Phi:\Mx\rightarrow L^2(\Omega)$ is the forward operator \eqref{eq_IP_offgrid}, and $\omega$ is an additive noise component, typically describing white Gaussian noise. 
The variational formulation of the problem of retrieving an estimate $\mu$ from $y$ is 
\begin{equation}
\argmin_{\mu\in\Mx}~T_\lambda(\mu)\quad\text{with} \quad T_\lambda(\mu):=\frac{1}{2}\|\Phi\mu-y\|^2+\lambda |\mu|(\Omega), \quad \lambda>0 \tag{$L^2-|\cdot|$}
    \label{eq:blasso}
\end{equation}
which is the minimisation problem of the Beurling-LASSO (BLASSO) functional, named after the work of the mathematician Beurling  \cite{beurling1938}. It is considered the generalisation of the discrete LASSO variational problem
\begin{equation}
    \argmin_{x\in\R^N} \frac{1}{2}\|\Tilde{\Phi} x-y\|^2+\lambda\|x\|_1
    \label{eq_lasso},
    \tag{$L^2-L^1$}
\end{equation}
where $\Tilde{\Phi}\in\R^{M\times N}$ is a discretisation of $\Phi:\Mx\rightarrow\mathcal{H}$ and $y\in\R^M$.
In \cite{bredies2013inverse}, the functonal $T_\lambda:\Mx\rightarrow\R$
is proved to be proper, convex and coercive, which guarantees the existence of solutions of \eqref{eq:blasso} and uniqueness under injectivity  on $\Phi$.

By optimality, $\mu_\lambda$ minimises $T_\lambda$ if and only if $0\in\partial T_\lambda(\mu_\lambda)$, i.e.
\begin{equation}
     0\in\partial T_\lambda(\mu_\lambda)=\Phi^*(\Phi\mu_\lambda-y)+\lambda\partial|\mu_\lambda|(\Omega) \, \iff \, \frac{1}{\lambda} \Phi^*(y-\Phi\mu_\lambda)\in\partial|\mu_\lambda|(\Omega),
\end{equation}
which can be written as 
\begin{equation}
   \eta\in\partial |\mu_\lambda|(\Omega)\qquad \text{with} \qquad \eta:=\frac{1}{\lambda} \Phi^*(y-\Phi\mu_\lambda),
    \label{eq_certificate}
\end{equation}
where $\eta$ is the so-called dual certificate of \eqref{eq:blasso}, see \cite{duval2014exact}.
The dual certificate formally generalises the concept of Euler equation, playing a crucial role in the characterisation of optimality conditions for \eqref{eq:blasso} and in devising optimisation algorithms in this setting, as better specified in the following. If $\mu_\lambda$ is a finite linear combination of Dirac masses \eqref{eq_discrete_measures}, recalling \eqref{eq_subdiff_TV_discrete}, we get that optimality conditions \eqref{eq_certificate} become simply 
\begin{equation}
    \eta(x_i)=\sign(a_i) \quad \wedge \quad \|\eta\|_\infty\le 1.
    \label{eq_opt_cond_blasso}
\end{equation}

Note that the optimality conditions \eqref{eq_opt_cond_blasso} can be also derived by studying the dual problem of \eqref{eq:blasso}:
\begin{equation}    \argmax_{\|\Phi^*p\|_{\infty,\Omega}\le 1} \langle y,p\rangle -\frac{\lambda}{2}\|p\|^2,
    \label{dual_probl_blasso}
\end{equation}
which provides further insights to the meaning of the dual certificate $\eta$. 
Given $\mu_\lambda\in\Mx$ solution of the minimisation problem \eqref{eq:blasso} and denoting by $p_\lambda\in L^2(\Omega)$ the solution of the dual problem \eqref{dual_probl_blasso}, the following  extremality conditions hold true
\begin{equation}
    \begin{cases}
        &\Phi^*p_\lambda\in\partial |\mu_\lambda|(\Omega),\\
        &-p_\lambda = \frac{1}{\lambda}\left(\Phi \mu_\lambda -y\right)
    \end{cases}.
    \label{extr_cond_blasso}
\end{equation}
From \eqref{extr_cond_blasso}, we retrieve in fact the optimality conditions \eqref{eq_opt_cond_blasso} expressed in terms of the dual certificate \eqref{eq_certificate}. In addition, when the dual certificate satisfies \eqref{eq_opt_cond_blasso}, by duality we have $\eta = \Phi^*p_\lambda$ with $p_\lambda$ being a solution of \eqref{dual_probl_blasso}.
Optimality conditions thus fully characterise the solution(s) of the BLASSO problem \eqref{eq:blasso}. They are indeed crucial in devising algorithms for its minimisation and, in particular, in the definition of good stopping criterions.

\section{Off-the-grid Poisson inverse problems}

We now present the signal-dependent Poisson modelling studied in this work, which differs from \eqref{eq_ip_offthegrid_gaussian} as it is not additive. 

We recall that in the case of a finite-dimensional image and measurement setting, where a linear forward operator $A\in\mathbb{R}^{M\times N}$ is such that 
\begin{equation} \label{eq:Poisso_discrete_posOp}
Ax\geq 0\quad\forall x\in\mathbb{R}^N_{\geq 0}
\end{equation}
and is used in combination with a positive background term $\epsilon\in\mathbb{R}^m_{>0}$, the non-negativity assumption
\begin{equation} \label{eq:Poisson_discrete}
\mathbb{R}^m_{\geq 0}\ni y = \mathcal{P}(Ax+\epsilon)
\end{equation}
means that for each $m=1,\ldots,M$, the $m-$th element $y_m$ is a random realisation of a uni-variate Poisson random variable with mean $(Ax)_m + \epsilon_m$ depending, in particular, on the unknown vector $x$\footnote{Note that the extension from a discrete support to the real-line can be done using the Gamma function.}. Moving from a discrete to a \red{continuous} settings requires some attention, since there is no clear infinite-dimensional interpretation of \eqref{eq:Poisson_discrete}.

First, let us consider $\Phi:\Mx\rightarrow L^2(\Omega)$ to be the forward operator in \eqref{eq_IP_offgrid} with $\mathcal{H}=L^2(\Omega)$ and let assume that $\Phi$ has  the following property (analogue to \eqref{eq:Poisso_discrete_posOp}) 
\begin{equation}
    \mu\in\Mxpos \text{ positive measure } \Rightarrow \Phi\mu(x)\ge 0 \quad \text{a .e. } x\in\Omega,
    \label{eq_phi_pos_def}
\end{equation}
that is $\Phi(\Mxpos)=L^2(\Omega)^+$, where $L^2(\Omega)^+=\{f\in L^2(\Omega) \text{ such that } f(x)\ge0 \text{ a .e. } x\in\Omega\}$. With a slight abuse of notation, we will say that $\Phi$ is a positive operator in the sense specified by \eqref{eq_phi_pos_def}.
Observe that whenever the measurement kernel $\varphi$ \eqref{eq_IP_offgrid} is non-negative, this always holds. This is the case, for instance, of standard image deblurring problems.

After introducing a strictly positive background correction term $b\in L^2(\Omega)^+$, we now want to interpret $y$ as an element of an infinite-dimensional function space, so that $y\in L^2(\Omega)$. This may sound not natural given the discrete support property of the Poisson distribution. However, as it has been done in several previous works, see, e.g., \cite{Bertero2009ImageDW,Bertero2010ADP,LiShen2015,Sawatzky2013,Lanza2014}, it is quite natural to extend Poisson variables to be element of real function spaces to allow  finer analysis. Note that a different approach covering a discrete measurement space for Poisson measurements  in the framework of grid-less reconstructions is considered in \cite{Pouchol2024}.

A common choice as a data term in the presence of Poisson data is the Kullback-Leibler divergence which we define in the following on $L^2(\Omega)^+$.
\begin{definition}
    The Kullback-Leibler divergence $\kl:L^2(\Omega)^+\times L^2(\Omega)^+\longrightarrow\R$ is defined by 
\begin{equation} 
    \kl(s,t):=\int_\Omega s(x)-t(x)+t(x)\Big(\log(t(x))-\log(s(x))\Big) \ \diff x. 
    \label{eq_kl}
\end{equation}
\end{definition}
Note that to consider $\kl(\Phi\mu+b,y)$ using the definition above, two issues have to be considered:
\begin{itemize}
    \item if $\mu$ is not a positive measure, then $\Phi\mu+b$ might not be positive;
    \item the noisy acquisition $y$ might still vanish in a non negligible region of the domain $\Omega$. This is due to the fact that a Poisson random variable with mean $\alpha$ vanishes with positive probability equal to $e^{-\alpha}$. For \eqref{eq_kl} to be well defined, it is thus required that $y>0$ almost everywhere.
\end{itemize}
To solve the first issue, we introduce the function
$\dkl:L^2(\Omega)\times L^2(\Omega)^+\longrightarrow\R\cup\{+\infty\}$ which extends \eqref{eq_kl} as 
\begin{equation}
\dkl(s,t)=
    \begin{cases}
        \kl(s,t) & s\in L^2(\Omega)^+\\
        +\infty & s\not\in L^2(\Omega)^+
    \end{cases}.
    \label{eq_dkl}
\end{equation}
Moreover, we just restrict our study to positive acquisitions $y$, thus requiring
\begin{equation}
    y\in L^2(\Omega)^+.
    \label{eq_y_pos}
\end{equation}
Under assumptions \eqref{eq_phi_pos_def} and \eqref{eq_y_pos} and using \eqref{eq_dkl}, the quantity $\dkl(\Phi\mu+b,y)$ is therefore well-defined for all $\mu\in\Mx$.

We thus consider the following grid-less Poisson reconstruction model:
\begin{equation}
   \boxed{ \argmin_{\mu\in\Mx} ~\dkl(\Phi\mu+b,y)+\lambda|\mu|(\Omega)+\iota_{\{\Mxpos\}}(\mu), \quad \lambda>0,
    \tag{$\dkl-|\cdot|$}}
    \label{poissonblasso}
\end{equation}
where the Poisson fidelity \eqref{eq_dkl} is coupled with the TV norm, together with the indicator function of the positive measures $\Mxpos$ defined by
\begin{equation}
\iota_{\{\Mxpos\}}(\mu)=\begin{cases}
        0 & \mu\ge 0\\
        +\infty & \text{otherwise}
    \end{cases},
    \label{ind_funct}
\end{equation}
to ensure $\mu$ to be non-negative. 

Following \cite{bredies2013inverse}, where the existence and uniqueness of the solution of \eqref{eq:blasso} is proved under certain conditions, we can state here a similar result for \eqref{poissonblasso}.
\begin{prop}
\label{poissonblasso_uniqueness}
    The minimisation problem \eqref{poissonblasso} admits a solution $\hat{\mu}\in\Mxpos$ if $\Phi:\Mx\longrightarrow L^2(\Omega)$ is weak* continuous. Moreover, the solution is unique if $\Phi$ is injective.
\end{prop}
\begin{proof}
    Observe that the functional $T_\lambda^{KL}(\mu):=\dkl(\Phi\mu+b,y)+\lambda|\mu|(\Omega)+\iota_{\{\Mxpos\}}(\mu)$ is proper and coercive.
    Moreover, the mapping $w\mapsto \dkl(w,y)$ is convex and sequentially weak* lower semi continuous. The norm in $\Mx$ and the indicator function $\mu\mapsto \iota_{\{\Mxpos\}}(\mu)$ are known to be sequentially lower semi-continuous in the weak* sense. All these considerations together yield the sequential weak* lower semi continuity of $T_\lambda^{KL}$. Hence, a minimising argument $\hat{\mu}\in\Mxpos$ does exist. Finally, an injective $\Phi$ results in a strictly convex $T_\lambda^{KL}$, which immediately gives the claimed uniqueness.
\end{proof}

\subsection{Dual problem and optimality conditions}
We analyse in the following the dual problem \cite{rockafellar1970convex,Ekeland1999} of \eqref{poissonblasso} and provide an analytical expression of the convex conjugate of the involved functions.  

The study of the dual problem of  the problem \eqref{poissonblasso} requires the computation of the convex conjugate $F^*:\mathscr{C}(\Omega,\R)\rightarrow\R\cup\{+\infty\}$ of the penalty term
\begin{equation}
    F:\Mx\rightarrow\R\cup\{+\infty\}, \quad F(\cdot):=|\cdot|(\Omega)+\iota_{\{\Mxpos\}}(\cdot),
    \label{eq_penalty_poisson}
\end{equation}
where $\Mx$ is endowed with the weak* topology,
and of the convex conjugate $G^*:L^2(\Omega)\rightarrow\R\cup\{+\infty\}$ of the fidelity term 
\begin{equation}
    G:L^2(\Omega)\rightarrow\R\cup\{+\infty\}, \quad G(\cdot):=\frac{1}{\lambda}\dkl(\cdot,y).
    \label{eq_fidelity_poisson_blasso}
\end{equation}

To compute the convex conjugate of the Kullback-Leibler functional \eqref{eq_fidelity_poisson_blasso}, we start  considering 
 the one-dimensional Kullback-Leibler function, defined by
\begin{equation}
    g_t(s)=\frac{1}{\lambda}\big(s-t+t\log(t)-t\log(s)\big), \qquad s,t>0 \text{ and } \lambda>0.
\end{equation}
Applying the definition of convex conjugate to $g_t$ yields
\begin{align*}
    g_t^*(s^*)&=\sup_{s>0}~ss^*-g_t(s)= \sup_{s>0}~ss^*-\frac{1}{\lambda}\Big(s-t+t\log(t)-t\log(s)\Big)=\\
    &=\sup_{s>0} ~\underbrace{s\Big(s^*-\frac{1}{\lambda}\Big)+\frac{t}{\lambda}\log(s)+\frac{t}{\lambda}-\frac{t}{\lambda}\log(t)}_{h(s)}.
\end{align*}
We have two cases:
\begin{itemize}
    \item[(i)] If $s^*\ge\frac{1}{\lambda}$, then $\lim_{s\to +\infty}h(s)=+\infty$ implies $\sup_{s>0} h(s)=+\infty\Rightarrow g_t\big(s^*\big)=+\infty$ for all $t>0$.
    \item[(ii)] If $s^*<\frac{1}{\lambda}$, then $\lim_{s\to +\infty}h(s)=\lim_{s\to 0^+}h(s)=-\infty$. Thus, being $h$ a convex and differentiable function its supremum is attained at $\hat{s}$ such that $h'(\hat{s})=0$, which can be computed by 
    \begin{align*}
    & h'(\hat{s})=s^*-\frac{1}{\lambda}+\frac{t}{\lambda \hat{s}}=\frac{\lambda \hat{s}s^*-\hat{s}+t}{\lambda \hat{s}}=0\\
   & \iff \lambda \hat{s}s^*-\hat{s}+t=0 \iff \hat{s}=\frac{t}{1-\lambda s^*}.
    \end{align*}
    Thus, 
    \begin{equation}
        g_t^*(s^*)=h\Big(\frac{t}{1-\lambda s^*}\Big)=-\frac{t}{\lambda}\log(1-\lambda s^*).
    \end{equation}
    Observe that $g_t^*(s^*)$ is well defined since $s^*<\frac{1}{\lambda}$. 
\end{itemize}
Hence, the convex conjugate $g_t^*$ of $g_t$ is 
\begin{equation}
    g^*_t(s^*)=
    \begin{cases}
        +\infty & s^*\ge \frac{1}{\lambda}\\
        -\frac{t}{\lambda}\log(1-\lambda s^*) & s^*<\frac{1}{\lambda}
    \end{cases}\,.
    \label{conv_conj_dkl_1d}
\end{equation}
Since $\dkl(\cdot,t)$ is defined also for non-positive functions, its one-dimensional counterpart is given by  $\Tilde{g}_t:\R\longrightarrow\R\cup\{+\infty\}$ where
\begin{equation}
    \Tilde{g}_t(s)=
    \begin{cases}
        g_t(s) & s>0\\
        +\infty & s\le0
    \end{cases} \,,
\end{equation}
so that its convex conjugate coincides with \eqref{conv_conj_dkl_1d}.
We thus have the following lemma.
\begin{lemma}
Let $G:L^2(\Omega)\rightarrow\R\cup\{+\infty\}$ be the function defined by 
\begin{equation}
    G(\cdot):=\frac{1}{\lambda}\dkl(\cdot,y),
\end{equation}
where $\dkl$ is given by \eqref{eq_dkl}.  The convex conjugate of $G$ is given by $G^*:L^2(\Omega)\rightarrow \R\cup\{+\infty\}$ defined as
\begin{align}
    G^*(s^*)=\begin{cases}
        \langle -\frac{y}{\lambda},\log(\mathbf{1}-\lambda s^*)\rangle & s^*(t)< \frac{1}{\lambda} \ \text{a.e.}\\
        +\infty & \text{otherwise}\\
    \end{cases}
    \label{eq_conv_conj_dkl}
\end{align}
where $\langle -\frac{y}{\lambda},\log(\mathbf{1}-\lambda s^*)\rangle=\int_\Omega -\frac{y(t)}{\lambda}\log(1-\lambda s^*(t))\diff t \in\R$, and $\mathbf{1}$ denotes the map defined by: $t\mapsto 1$ a.e.

\end{lemma}
\begin{proof}
    The computation of \eqref{eq_conv_conj_dkl} follows straightforwardly from the 1-dimensional case given by \eqref{conv_conj_dkl_1d}. In particular, we observe that $G(s)=\int_\Omega \Tilde{g}_{y(x)}\big(s(x)\big)\diff x$ and, thanks to a result from \cite[Prop.IX.2.1]{Ekeland1999}, we can conclude that $G^*(s^*)=\int_\Omega \Tilde{g}_{y(x)}^*\big(s^*(x)\big)\diff x$. Indeed, the function $G$ satisfies the hypothesis of the proposition in \cite{Ekeland1999} by considering $G(y)=\frac{1}{\lambda}\dkl(y,y)=0<+\infty$.
\end{proof}

We compute now the convex conjugate of $F$ defined in \eqref{eq_penalty_poisson}.

\begin{lemma}
    For $\psi\in\mathscr{C}(\Omega,\R)$, its convex conjugate $F^*:\mathscr{C}(\Omega,\R)\rightarrow\R\cup\{+\infty\}$  is defined by 
    \begin{equation}
    F^*(\psi)=\begin{cases}
        0 & \text{if} \ \psi(x)\le 1 \quad \forall x\in\Omega\\
        +\infty & \text{otherwise}
    \end{cases}.
    \label{eq_conv_conj_penalty_poisson}
\end{equation}
\end{lemma}
\begin{proof}
By definition of convex conjugate \cite{Combettes2011,rockafellar1970convex}, for any $\psi\in\mathscr{C}(\Omega,\R)$ we write
    \begin{align}
        F^*(\psi)&=\sup_{\mu\in\Mx} \langle \psi,\mu\rangle_{\mathscr{C}(\Omega,\R)\times\Mx} -|\mu|(\Omega)-\iota_{\Mxpos}(\mu)\\
        &=\sup_{\mu\in\Mxpos} \langle \psi,\mu\rangle_{\mathscr{C}(\Omega,\R)\times\Mx} -|\mu|(\Omega)\\
        &\ge \langle \psi,\mu\rangle_{\mathscr{C}(\Omega,\R)\times\Mx} -|\mu|(\Omega) \quad \forall \mu\in\Mxpos.
    \end{align}
    If there exists $\bar{x}\in\Omega$ such that $\psi(\bar{x})>1$, by taking $\bar{\mu}=\alpha \delta_{\bar{x}}$ with $\alpha>0$ we obtain
    \begin{align}
        F^*(\psi)&\ge \langle \psi,\bar{\mu}\rangle_{\mathscr{C}(\Omega,\R)\times\Mx} -|\bar{\mu}|(\Omega)\\
        &=\alpha\psi(\bar{x})-\alpha=\alpha\Big(\psi(\bar{x})-1\Big),
    \end{align}
   and the limit for $\alpha\rightarrow +\infty$ of the latter inequality yields $F^*(\psi)=+\infty$. 

   Assume now $\psi(x)\le 1$ for all $x\in\Omega$. We observe that, since $\psi(x)\le 1$, for any positive $\mu\in\Mxpos$
   \begin{align}
       &\langle \psi,\mu\rangle_{\mathscr{C}(\Omega,\R)\times\Mx}=\int_\Omega \psi(x)\diff\mu(x)\le\int_\Omega 1 \diff \mu(x)=|\mu|(\Omega) \quad \forall\mu\in\Mxpos \\
      \Rightarrow & \langle\psi,\mu\rangle_{\mathscr{C}(\Omega,\R)\times\Mx}-|\mu|(\Omega)\le 0 \quad \forall\mu\in\Mxpos\\
      \Rightarrow &\sup_{\mu\in\Mxpos} \langle \psi,\mu\rangle_{\mathscr{C}(\Omega,\R)\times\Mx} -|\mu|(\Omega) \le 0 \Rightarrow F^*(\psi)\le 0.
   \end{align}
   Moreover, for $\mu\equiv 0$ we have $\langle \psi,\mu\rangle_{\mathscr{C}(\Omega,\R)\times\Mx} -|\mu|(\Omega)=0$. Hence, $F^*(\psi) = 0$. This concludes the proof of \eqref{eq_conv_conj_penalty_poisson}.
\end{proof}

\subsection{Dual problem}

The study of dual problems of sparse-regularisation models with non-negativity constraints has been carried out in \cite{Ndiaye2017,Wang2014} in the discrete setting of LASSO, and in \cite{Dantas2021} for the discrete counterpart of \eqref{poissonblasso}, that is with the Kullback-Leibler divergence as fidelity and the $L^1$ penalty. In the following we compute the dual problem of \eqref{poissonblasso} in $\Mx$, endowed with the weak* topology, by means of $F^*$ \eqref{eq_conv_conj_penalty_poisson} and $G^*$ \eqref{eq_conv_conj_dkl}, \red{exploiting results on duality that can be found in \cite{ekeland1999convex} and that are similarly applied in the context of BLASSO in \cite{Laville2021,poon2019sparse}. For readability purpose, we report the statement of the needed result with coherent notation.}

\begin{lemma}
    \red{Let $V$ be a locally convex vector space and let $Y$ be a Banach space.} For $\Lambda:V\rightarrow Y$ linear and continuous operator, $F:V\rightarrow\R$ and $G:Y\rightarrow \R$ convex functionals, we consider the following primal problem:
    \begin{equation}
        \argmin_{u\in V}~ F(u)+G(\Lambda u).
        \label{eq_primal_problem_general}
    \end{equation}
    The corresponding dual problem reads
    \begin{equation}
    \argmax_{p\in Y^*}-F^*(\Lambda^*p)-G^*(-p),
    \label{eq_dual_problem_general}
\end{equation}
where $\Lambda^*:Y^*\rightarrow V^*$ is the adjoint operator of $\Lambda$ and $F^*:V^*\rightarrow\R\cup\{+\infty\}$, $G^*:Y^*\rightarrow\R\cup\{+\infty\}$ are the convex conjugate of $F$ and $G$.

Moreover, if $u\in V$ and $p\in Y^*$ are respectively solutions of the primal \eqref{eq_primal_problem_general} and dual \eqref{eq_dual_problem_general} problems, the following extremality conditions hold: 
\begin{equation}
    \begin{cases}
        &\Lambda^*p\in\partial F(u)\\
        &-p\in\partial G(\Lambda u)
    \end{cases}.
    \label{extr_cond_gen}
\end{equation}
\label{lemma_dual_probl}
\end{lemma}

By Lemma \ref{lemma_dual_probl},  the dual problem of \eqref{poissonblasso} can be obtained by plugging \eqref{eq_conv_conj_dkl} and \eqref{eq_conv_conj_penalty_poisson} into \eqref{eq_dual_problem_general}. We thus have:
\begin{align}
  &\argmax_{p\in L^2(\Omega)} ~-F^*(\Phi^*p)-G^*(-p)\\
  =&\argmax_{p\in L^2(\Omega)}~ -F^*(\Phi^*p)+\begin{cases}
       \langle\frac{y-b}{\lambda},\log(\mathbf{1}+\lambda p)\rangle & -p(x)<\frac{1}{\lambda} \quad \text{a.e.}  \ x\in\Omega \\
       -\infty & \text{otherwise}  
    \end{cases}\\
    =&\argmax_{p\in L^2(\Omega)} ~-F^*(\Phi^*p)+\begin{cases}
        \langle\frac{y-b}{\lambda},\log(\mathbf{1}+\lambda p)\rangle & p(x)>-\frac{1}{\lambda} \quad \text{a.e.}  \ x\in\Omega\\
         -\infty & \text{otherwise} \\
    \end{cases}\\
    =&\argmax_{p\in L^2(\Omega)\text{ s.t. }p>-\frac{1}{\lambda}} -F^*(\Phi^*p) + \left\langle\frac{y-b}{\lambda},\log(\mathbf{1}+\lambda p)\right\rangle\\
    =&  \argmax_{p\in L^2(\Omega)\text{ s.t. }p>-\frac{1}{\lambda}} \begin{cases}
        \;0 & \forall x\in\Omega, \;\, \Phi^*p(x)\le 1\\
        -\infty & \exists x\in\Omega, \;\, \Phi^*p(x)>1
    \end{cases}\quad + \quad \left\langle\frac{y-b}{\lambda},\log(\mathbf{1}+\lambda p)\right\rangle\\
    =&  \argmax_{p\in \mathcal{D}}~ \left\langle\frac{y-b}{\lambda},\log(\mathbf{1}+\lambda p)\right\rangle \,, 
    \label{dual_poissonblasso}
\end{align}
where $\mathcal{D}=\{ p\in L^2(\Omega): p(x)>-\frac{1}{\lambda} \; \text{a.e.}  \ x\in\Omega \text{ and }  \Phi^*p(x)\le 1 \ \forall x\in\Omega\}$.

\subsection{Extremality conditions}

By Fenchel-Rockafellar duality (Lemma \ref{lemma_dual_probl}), extremality conditions \eqref{extr_cond_gen} can be obtained. Given $\mu_\lambda\in\Mx$ solution of the primal problem \eqref{poissonblasso} with regularisation parameter $\lambda>0$ and $p_\lambda\in L^2(\Omega)$ solution of the dual problem \eqref{dual_poissonblasso}, they read
\begin{equation}
    \begin{cases}
       \Phi^*p_\lambda\in {\partial F(\mu_\lambda)=\partial\Big(|\mu_\lambda|(\Omega)+ \iota_{\Mxpos}(\mu_\lambda)\Big)}\\
       -p_\lambda\in\frac{1}{\lambda}\partial_1\dkl(\Phi\mu_\lambda+b,y)=\frac{1}{\lambda}\Big(\mathbf{1}-\frac{y}{\Phi\mu_\lambda+b}\Big)
    \end{cases},
    \label{eq_extr_cond_poisson}
\end{equation} where by $\mathbf{1}$ we denote again the map defined by: $t\mapsto 1$ a.e. We remark that the notation $\partial_1\dkl(\cdot,\cdot)$ denotes the subdifferential of $\dkl(\cdot,\cdot)$ computed with respect to the first variable and it is given by the following expression
\begin{equation}
    \partial_1\tilde{\mathcal{D}}_{KL}(s,t):=\begin{cases}
    \mathbf{1}-\frac{t}{s} & s\in L^2(\Omega)^+\\
    \emptyset & s\not\in L^2(\Omega)^+
    \end{cases}.
\end{equation}
When evaluated in $(\Phi\mu_\lambda+b,y)$, since $\Phi\mu_\lambda+b$ is always positive, the subdifferential is always non-empty.
\begin{oss}
If $\mu_\lambda\in\Mx$ is solution of the primal problem \eqref{poissonblasso} and $p_\lambda\in L^2(\Omega)$ is solution of the dual problem \eqref{dual_poissonblasso}, then from \eqref{eq_extr_cond_poisson} we have 
\begin{equation}
    -p_\lambda=\frac{1}{\lambda}\Big(\mathbf{1}-\frac{y}{\phi\mu_\lambda+b}\Big)\quad \Rightarrow  \quad p_\lambda=\frac{y-\Phi\mu_\lambda-b}{\lambda(\Phi\mu_\lambda+b)}.
    \label{p_lambda}
\end{equation}
It follows that ${p_\lambda>-\frac{1}{\lambda} \ {a.e.}\ \iff y>0}$, which holds by hypothesis \eqref{eq_y_pos}.  
\end{oss}

\medskip
\red{To have a complete analytical expression of the extremality conditions \eqref{eq_extr_cond_poisson}, we compute in the following proposition the subdifferential $\partial\big(|\mu|(\Omega)+\iota_{\Mxpos}(\mu)\big)$.}
\begin{prop}
\label{prop_partialF}
    The subdifferential of the penalty term $F$ in \eqref{eq_penalty_poisson} can be directly computed for all $\mu\in\Mx$ as
    \begin{equation}
        \partial F(\mu)=\begin{cases}
            \emptyset & \mu \not\in \Mxpos\\
            \{\psi \in \mathscr{C}(\Omega,\R)|\psi(x)\le 1 \ \forall x\in\Omega \text{ and } \psi(x)=1 \ \forall x\in\text{supp}(\mu)\} & \mu\in\Mxpos
        \end{cases}.
        \label{eq_partial_penalty}
    \end{equation}
\end{prop}
\begin{proof}
    We start by recalling the definition of subdifferential:
    \begin{equation}
        \partial F(\mu)=\{\psi\in\mathscr{C}(\Omega,\R)| F(\bar{\mu})\ge F(\mu)+\langle \psi,\bar{\mu}-\mu\rangle \ \forall \bar{\mu}\in\Mx\}.
        \label{eq_partial_penalty_def}
    \end{equation}
    If $\mu\not\in\Mxpos$, then $F(\mu)=+\infty$ and hence the inequality of \eqref{eq_partial_penalty_def} will never be satisfied. Thus, $\partial F(\mu)=\emptyset$.

    We consider now a positive measure $\mu\in\Mxpos$. In this case, the penalty term $F$ reduces to $F(\mu)=|\mu|(\Omega)=\int_\Omega 1 \diff \mu=\langle 1,\mu\rangle$. If we take a measure $\bar{\mu}\not\in\Mxpos$ in \eqref{eq_partial_penalty_def}, we would have $F(\bar{\mu})=+\infty$ and the inequality would be automatically verified. Hence, in \eqref{eq_partial_penalty_def}, it is equivalent to require the inequality to be verified for all $\bar\mu\in\Mxpos$. Let now $\bar{\mu}\in\Mxpos$, for which we write $F(\bar{\mu})=\int_\Omega 1 \diff \bar{\mu}=\langle 1, \bar{\mu}\rangle$. Hence, \eqref{eq_partial_penalty_def} becomes
    \begin{align}
        \partial F(\mu) &= \{\psi\in\mathscr{C}(\Omega,\R)| \langle 1, \bar{\mu}\rangle\ge \langle 1, \mu\rangle+\langle \psi,\bar{\mu}-\mu\rangle \ \forall \bar{\mu}\in\Mxpos\} \\
        &=\{\psi\in\mathscr{C}(\Omega,\R)|\langle 1,\bar{\mu}-\mu\rangle\ge \langle \psi,\bar{\mu}-\mu\rangle \ \forall \bar{\mu}\in\Mxpos\}.
    \end{align}
    It is easy to deduce that $\psi(x)\le 1$ for all $x\in\Omega$. Indeed, by taking $\bar{\mu}=\mu+\delta_x$ with $x\in\Omega$, we have 
    \begin{equation}
        1=\int_\Omega 1 \diff \delta_x=\int_\Omega 1 \diff (\bar{\mu}-\mu)=\langle 1,\bar{\mu}-\mu\rangle\ge \langle \psi,\bar{\mu}-\mu\rangle=\int_\Omega \psi \diff\delta_x=\psi(x).
    \end{equation}
    By taking $\bar{\mu}=0$, we have
    \begin{align}
        & \langle 1,-\mu\rangle\ge \langle \psi,-\mu\rangle \iff 
         \langle \psi,\mu\rangle \ge \langle 1,\mu\rangle \iff 
          \int_\Omega \psi \diff \mu \ge \int_\Omega 1 \diff \mu,
    \end{align}
    that is true if and only if $\psi(x)\ge 1$ for all $x\in\text{supp}(\mu)$. Hence, $\psi(x)=1$ on the support of $\mu$.

    We showed that if $\psi\in\partial F(\mu)$ then $\psi(x)\le 1$ for every $x\in\Omega$ and $\psi(x)=1$ on the support of $\mu$. Showing the other inclusion is straightforward. Let $\psi\in\mathscr{C}(\Omega,\R)$ such that $\psi(x)\le 1$ for every $x\in\Omega$ and $\psi(x)=1$ on the support of $\mu$. Let $\bar\mu\in\Mxpos$. We need to show that $\psi\in\partial F(\mu)$, hence it satisfies the inequality of \eqref{eq_partial_penalty_def}. We observe that 
    \begin{align}
        & \langle\psi,\bar\mu\rangle = \int_\Omega \psi\diff \bar\mu\le\int_\Omega 1\diff\bar\mu=\langle 1,\bar\mu\rangle\\
        & \langle\psi,\mu\rangle = \int_\Omega \psi\diff \mu=\int_\Omega 1\diff\mu=\langle 1,\mu\rangle
    \end{align}
    and, hence,
    \begin{equation}
        \langle \psi,\bar\mu-\mu\rangle=\langle\psi,\bar\mu\rangle-\langle\psi,\mu\rangle\le {\color{red}\langle 1,\bar\mu\rangle-\langle 1,\mu\rangle}.
    \end{equation}
    This concludes the proof.
\end{proof}
\red{
\begin{oss}
    Recall that, in general, for two proper convex functions $f_1,f_2:\CX\rightarrow \R\cup\{+\infty\}$ there holds
$\partial( f_1+f_2)\subseteq \partial f_1+\partial f_2$. 
Note that by direct computation there holds
\begin{equation}
    \partial F(\cdot)= \partial |\cdot|(\Omega)+\partial \iota_{\Mxpos} (\cdot).
\end{equation}
Indeed, the subdifferential of the TV norm is
\begin{equation}
    \partial |\mu|(\Omega)=\{\psi\in\mathscr{C}(\Omega,\R)|\|\psi\|_\infty \le 1 \ \text{and} \ \int_\Omega \psi\diff\mu=|\mu|(\Omega)\}
\end{equation}
and, the subdifferential of $\iota_{\Mxpos}$ can be computed similarly as in Proposition \ref{prop_partialF} and is equal to
\begin{equation}
    \partial \iota_{\Mxpos} (\mu)=\begin{cases}
        \emptyset & \mu\notin \Mxpos\\
        \{\psi\in\mathscr{C}(\Omega,\R)|\psi(x)\le 0 \ \forall x\in\Omega \ \text{and} \ \psi(x)=0\ \forall x\in\text{supp}(\mu)\} & \mu\in\Mxpos
    \end{cases}.
\end{equation}

If $\mu\notin\Mxpos$, we have $\partial \iota_{\Mxpos} (\mu)=\emptyset$ and $\partial F(\mu)=\emptyset$. Hence, $\partial |\mu|(\Omega)+\partial \iota_{\Mxpos} (\mu)=\partial |\mu|(\Omega)+\emptyset=\emptyset$, which is equal to $\partial\big(|\mu|(\Omega)+\iota_{\Mxpos}(\mu)\big)=\partial F(\mu)=\emptyset$. 

If $\mu\in\Mxpos$, having $\int_\Omega \psi\diff\mu=|\mu|(\Omega)$ for $\psi\in\mathscr{C}(\Omega,\R)$ is equivalent to requiring $\psi(x)=1$ for all $x\in\text{supp}(\mu)$ since $|\mu|(\Omega)=\langle 1,\mu\rangle=\int_\Omega 1\diff\mu$. Thus, the expression of the subdifferential $\partial |\mu|(\Omega)$ reduces to
$$\partial |\mu|(\Omega)=\{\psi\in\mathscr{C}(\Omega,\R)| -1\le \psi(x)\le 1 \ \forall x\in\Omega \ \text{and} \ \psi(x)=1 \ \forall x\in\text{supp}(\mu)\}.$$
Thanks to the above expression, it is straightforward to show that $\partial F(\mu)=\partial\big(|\mu|(\Omega)+\iota_{\Mxpos}(\mu)\big)= \partial |\mu|(\Omega)+\partial \iota_{\Mxpos} (\mu)$ for all $\mu\in\Mx$.
\end{oss}
}

Thanks to \eqref{eq_partial_penalty} of the latter proposition, it is now possible to better characterise the extremality conditions \eqref{eq_extr_cond_poisson}, under the assumption that $\mu_\lambda$, solution of \eqref{poissonblasso}, is a discrete measure. Indeed, if $\mu_\lambda=\sum_{i=1}^{N_\lambda}(a_\lambda)_i\delta_{(x_\lambda)_i}$ we have that \eqref{eq_extr_cond_poisson} is equivalent to
\begin{equation}
    \Phi^*p_\lambda(x)\le 1 \quad \forall x\in\Omega \quad\text{and} \quad \Phi^*p_\lambda\big((x_\lambda)_i\big)=1,\ i=1,\ldots,N_\lambda,
   \label{extr_cond_poisson_final}
\end{equation}
with $p_\lambda$ solution of the dual problem \eqref{dual_poissonblasso}, explicitly given by \eqref{p_lambda}.
The quantity $\Phi^*p_\lambda$, similarly as for BLASSO, is referred to as dual certificate $\eta:=\Phi^*p_\lambda$ with $p_\lambda$ as defined in \eqref{p_lambda}.

\section{The Sliding Frank-Wolfe algorithm}
\label{sec_4_sfw}
Both problems \eqref{poissonblasso} and \eqref{eq:blasso} are defined over the space $\Mx$, an infinite dimensional non-reflexive Banach space. Due to non-reflexivity, it is not simple to define therein proximal-based algorithms, see \cite{Valkonen2023} for some recent attempt. 
Any solver for such problems shall take into account the infinite dimensional nature of $\Mx$: popularly used algorithms in this setting are semi-definite programming approaches (for Fourier measurements) \cite{FernandezGranda2013}, conditional gradient algorithms \cite{Frank1956,BrediesLorenz2009,Denoyelle2019} and particle gradient descent \cite{Chizat2018OnTG,Chizat2021}, which is an optimal-transport based algorithm. 
In this work, we will only focus on conditional gradient strategies, namely the Frank-Wolfe and Sliding Frank-Wolfe algorithms, see \cite{Pokutta} for a survey. 
\begin{algorithm}[t]
\caption{Sliding Frank-Wolfe (SFW) algorithm \cite{Denoyelle2019}}
\textbf{Initialisation:}  $\mu^{0}=0$.

\textbf{repeat} for $k=1,2,\ldots,K_\text{max}$
 \begin{itemize}
     \item[] $\mu^{k}=\sum_{i=1}^{N^{k}} a_i^{k} \delta_{x_i^{k}}, a_i^{k} \in \mathbb{R}, x_i^{k}\in\Omega$, find $x_*^{k} \in \Omega$ s.t.:
     \begin{equation}
x_*^{k} \in \argmax _{x \in \Omega}\left|\eta(\lambda,\mu^k)(x)\right| \quad \text{where}\quad \eta(\lambda,\mu^k) \quad \text{is defined by \eqref{eq_gen_dual_cert} with } \mu=\mu^k
\end{equation}
\item[] \textbf{if} $\left|\eta(\lambda,\mu^k)\left(x_*^{k}\right)\right| \le 1$ 
\begin{itemize}
    \item[] $\mu^{k}$ is a solution of \eqref{eq_gen_prob_hom} $\Rightarrow$ \textbf{stop}
\end{itemize}
\item[] \textbf{else}
\begin{itemize}
    \item \textbf{insertion step}: add support for the new spike and amplitudes' estimation
    \begin{align}
        &x^{k+1/2}=\left(x_1^{k}, \ldots, x_{Nk}^{k}, x_*^{k}\right) \notag\\
        & a^{k+1/2} \in \argmin _{a \in \mathbb{R}^{N^{k}+1}} f_{y^\delta,b}\left(\Phi_{x^{k+1/2}} a\right)+\lambda\|a\|_1 +\alpha \iota_{\ge0}(a)
        \label{sfw_est_ampl} \\
        & \text{Update:} \qquad \mu^{k+1/2}=\sum_{i=1}^{N^{k}} a_i^{k+1/2} \delta_{x_i^{k}}+a_{N^k+1}^{k+1/2} \delta_{x_*^{k}}
    \end{align}
    \item \textbf{sliding step}: using a non-convex solver initialised with $\left(a^{k+1/2}, x^{k+1/2}\right)$
    \begin{align}
       & \left(a^{k+1}, x^{k+1}\right) \in \underset{(a, x) \in \mathbb{R}^{N^k+1} \times \Omega^{N^k+1}}{\arg \min } f_{y^\delta,b}\left(\Phi_x a\right)+\lambda\|a\|_1 +\alpha \iota_{\ge 0}(a) \label{eq_sliding_step}\\ 
       & \text{Update:} \qquad \mu^{k+1} = \sum_{i=1}^{N^{k}+1} a_i^{k+1}\delta_{x_i^{k+1}}
    \end{align}
    \item \textbf{pruning}: eventually remove zero amplitudes Dirac masses from $\mu^{k+1}$
\end{itemize}
 \end{itemize}
\textbf{until} convergence
\label{alg_sfw}
\end{algorithm}

In \cite{Denoyelle2019}, the authors detail how the conditional gradient algorithm, also known as Frank-Wolfe algorithm \cite{Frank1956}, can be used to minimise the BLASSO \eqref{eq:blasso} functional using an epigraphic lift, and then propose the Sliding version, which significantly improves the reconstruction quality by adding an extra step where both positions and amplitudes are re-optimised. Similar strategies are here used to minimise  \eqref{poissonblasso}. We report here the statement of the result without a proof as it follows from \cite[Lemma 4]{Denoyelle2019}.

\begin{lemma}
    The solution $\Bar{\mu}\in\Mx$ to \eqref{poissonblasso} is equivalent to the solution $\Bar{\mu}\in\Mx$ to the problem
\begin{equation}
    \argmin _{(t, \mu) \in C} ~\tilde{T}_\lambda^{KL}(\mu, t) \quad\text{with} \quad \tilde{T}_\lambda^{KL}(\mu, t):= \dkl(\Phi\mu+b,y)+\lambda t +\iota_{\{\Mxpos\}}(\mu),
    \label{eq_lift_blasso_poisson}
\end{equation}
where 
$C :=\left\{(t, \mu) \in \mathbb{R}^{+} \times \Mx ;|\mu|(\Omega) \le t \le M\right\} \quad \text{with } M := \frac{\dkl(b,y)}{\lambda}$.
\label{epigraphic_lif_blasso_poisson}
\end{lemma}

Given a convex fidelity functional $f_{y^\delta,b}:L^2(\Omega) \rightarrow\R\cup\{+\infty\}$ defined in terms of the observed data $y^\delta\in L^2(\Omega)$ and, potentially, a background term $b\in L^2(\Omega)^+$, we consider the following general optimisation problem:
\begin{equation}
    \argmin_{\mu\in\Mx} ~f_{y^\delta,b}(\Phi\mu)+\lambda|\mu|(\Omega)+\alpha \iota_{\Mxpos}(\mu), \qquad \lambda >0,\quad \alpha\in\{0,1\},
    \tag{$\mathcal{P}(\lambda)$}
    \label{eq_gen_prob_hom}
\end{equation}
which encompasses the BLASSO as well as problem \eqref{eq_gen_prob_hom} and
where $\alpha\in\{0,1\}$ may enforce non-negativity constraints. We assume in the following the fidelity functional $f_{y^\delta,b}$ to be smooth on $L^2(\Omega)^+$.

The pseudocode corresponding to the optimisation of the general problem \eqref{eq_gen_prob_hom} is reported in Algorithm \ref{alg_sfw}. Observe that the stopping criterion is expressed in terms of the dual certificate  $\eta(\lambda,\mu)\in L^2(\Omega)$ of the general problem \eqref{eq_gen_prob_hom}, which is defined in terms of the subgradient of the fidelity $\partial f_{y^\delta,b}$ and reads 
\begin{equation}
    \eta(\lambda,\mu)=\frac{1}{\lambda}\Tilde{\eta}(\mu)  \quad \text{with} \quad \Tilde{\eta}(\mu)\in\begin{cases}
        -\Phi^*\partial f_{y^\delta,b}(\Phi\mu) & \alpha=0\\
        \left(-\Phi^*\partial f_{y^\delta,b}(\Phi\mu)\right)_+ & \alpha=1
    \end{cases},
    \label{eq_gen_dual_cert}
\end{equation}
depending on the parameter $\alpha$ so that the dual certificate \eqref{eq_certificate} of \eqref{eq:blasso} and the dual certificate \eqref{extr_cond_poisson_final} of \eqref{poissonblasso} can be obtained by choosing $\alpha$ and $f_{y^\delta,b}$, respectively.
Observe that the subdifferential $\partial f_{y^\delta,b}$ is single-valued for BLASSO \eqref{eq:blasso}, since $f_{y^\delta,b}(w)=\frac{1}{2}\|w-y^\delta\|^2$ is smooth on $L^2(\Omega)$. When choosing the Kullback-Leibler $f_{y^\delta,b}(w)=\tilde{D}_{KL}(w+b,y^\delta)$ for \eqref{poissonblasso}, the subdifferential $\partial f_{y^\delta,b}$ is either single-valued on $L^2(\Omega)^+$ or empty. Hence, the dual certificate is always defined without ambiguity.
The generalised optimality condition for \eqref{eq_gen_prob_hom} reads
\begin{equation}
    \|\eta(\lambda,\mu)\|_\infty\le 1,
    \label{eq_opt_cond_gen}
\end{equation}
and, under the hypothesis that the solution is a finite linear combination of Diracs as \eqref{eq_discrete_measures}, the dual certificate of the solution satisfies $\eta(\lambda,\mu)(x_i)=\sign(a_i)$ where the points $x_i\in\Omega$ are the support of $\mu$.  

\section{Parameter selection via algorithmic homotopy}  \label{sec:homotopy}

\begin{figure}[t]
    \centering
    \begin{subfigure}[b]{0.3\textwidth}
        \includegraphics[width=\textwidth]{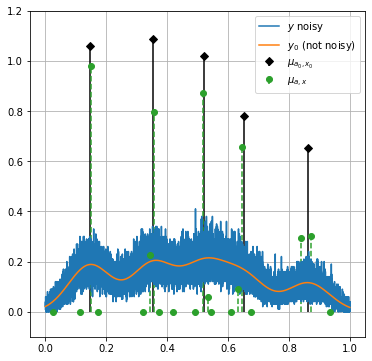}
        \caption{ $\lambda\ll 1$}
    \end{subfigure}
    \begin{subfigure}[b]{0.3\textwidth}
        \includegraphics[width=\textwidth]{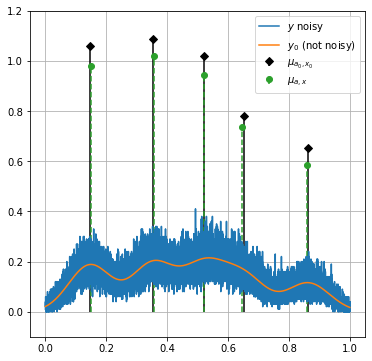}
        \caption{$\lambda=1$}
    \end{subfigure}
    \begin{subfigure}[b]{0.3\textwidth}
        \includegraphics[width=\textwidth]{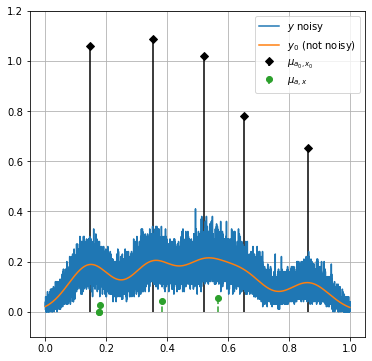}
        \caption{$\lambda\gg 1$}
    \end{subfigure}
    \caption{Reconstructions obtained using SFW for a 1D sparse deconvolution problem with Poisson noise. Ground truth spikes (black) and reconstructed ones (green) using Algorithm \ref{alg_sfw} for some choices of $\lambda$ are shown. When $\lambda\ll 1$, the number and intensities of spikes are overestimated, while for $\lambda\gg 1$} they are underestimated.
    \label{fig_lambda_sfw}
\end{figure}

The performance of Algorithm \ref{alg_sfw}
for solving \eqref{eq_gen_prob_hom} depends on the choice of the regularisation parameter $\lambda>0$: it plays indeed a fundamental role in the sparsity pattern of the solution and in the enforcement of the  stopping rule, $\max_{x\in\Omega} \left|\eta(\lambda,\mu^k)(x)\right|\le 1$. 
Namely, a high value of $\lambda$ forces only few iterations of the algorithm to be performed. Such choice impacts also  both the estimation and the sliding steps of Algorithm \ref{alg_sfw}, being it associated with the sparsity-promoting $L^1$ penalty used to compute the amplitude vector. On the other hand, smaller values of $\lambda$ provide a better data fit, with more spikes with higher intensities, see Figure \ref{fig_lambda_sfw}. 

\medskip
In \cite{Courbot2021}, the authors propose a method based on algorithmic homotopy \cite{Osborne2000,Osborne2000b} to choose an optimal regularisation parameter for the resolution of BLASSO \eqref{eq:blasso}.
The idea behind homotopy algorithms is to avoid the exploration of the whole  Pareto frontier by grid search, while providing an iterative procedure computing only few parameters up to a target value. More precisely, starting from an initial overestimated value $\lambda_1>0$, a solution $\mu_{\lambda_1}$ to \eqref{eq_gen_prob_hom} is computed by, e.g., Algorithm \ref{alg_sfw}. At each homotopy iteration, if the solution does not fit well the data up to some tolerance $\sigma_{\text{target}}$ depending on the noise magnitude $\delta$, then $\lambda$ is decreased. A new solution $\mu_{\lambda_2}$ to \eqref{eq_gen_prob_hom} is thus computed in the next homotopy step and so on.
Homotopy algorithms thus explore the Pareto frontier for a small set of values $\lambda$ and select its biggest value for which the solution to \eqref{eq_gen_prob_hom} meets a convergence criterion depending on $\delta$. In Figure \ref{fig_pareto_frontier}, one can see in red the discrete values produced by the homotopy strategy we are going to describe, which stops when the fidelity term goes under the value of $\sigma_{\text{target}}(\delta)$, in grey.

\begin{figure}[h!]
    \centering
    \includegraphics[width=0.44\textwidth]{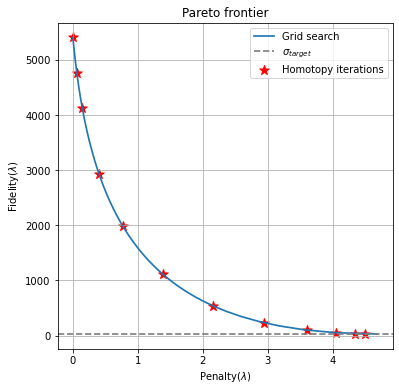}
    \caption{Homotopy algorithms explore the Pareto frontier iteratively for a strictly decreasing sequence of regularisation parameters $\lambda$. In blue: fine discretisation of the Pareto frontier with grid search. In red: homotopy iterations, corresponding to different values of $\lambda$. In gray: target value for the fidelity. }
    \label{fig_pareto_frontier}
\end{figure}

\begin{algorithm}[h!]
\textbf{Input:} $y\in L^2(\Omega)$, $b\in L^2(\Omega),\ b\ge 0$, $\Phi\in\mathcal{L}(\Mx,L^2(\Omega))$

\textbf{Parameters:} $\gamma\in(0,1)$, $c>0$, $\sigma_{\text{target}}>0$

\textbf{Output:} optimal $\hat{\mu}\in\Mx$

\textbf{Initialisation:} $\hat{\mu}_0\in\Mx$ and $\lambda_1=\gamma\left\|\eta(1,\hat{\mu}_0)\right\|_\infty$

\textbf{repeat} for $t=1,2,\ldots,T_\text{max}$
\begin{itemize}
        \item[1.] Compute $\hat{\mu}_t$ solution of \eqref{eq_gen_prob_hom} with $\lambda=\lambda_t$ with warm start $\mu_t^{[0]}=\hat{\mu}_{t-1}$.
        \item[2.] Compute $\sigma_t$ from the residual: 
        \begin{equation}
            \sigma_t=f_{y^\delta,b}(\Phi\hat{\mu}_t)
            \label{eq_sigmat_hom}
        \end{equation}
        \item[3.] \textbf{if} $\sigma_t<\sigma_{\text{target}}$
        \begin{itemize}
            \item[] $\hat{\mu}_t$ is a solution $\Rightarrow$ \textbf{stop}
        \end{itemize}
        \item[4.] \textbf{else if} $\sigma_t\ge\sigma_{\text{target}}$
        \begin{equation}
            \text{Update} \quad \lambda_{t+1}=\frac{\lambda_t \left\|\eta(\lambda_t,\hat{\mu}_t)\right\|_\infty}{c+1}
            \label{eq_alg_hom_lambdat}
        \end{equation}
    \end{itemize}
{\textbf{until}} $\sigma_t<\sigma_{\text{target}}$
\caption{Homotopy algorithm in $\Mx$}
\label{alg_homotopy}
\end{algorithm} 

In \cite{Courbot2021}, the Sliding Frank-Wolfe with homotopy is proposed for the  BLASSO problem \eqref{eq:blasso}. In the following, we extend the strategy in Algorithm \ref{alg_homotopy} for a general fidelity term within an off-the-grid setting. Note that each homotopy iteration $t\in\N$ performs the following steps:
\begin{itemize}
    \item Compute $\mu_{\lambda_t}:=\hat{\mu}_t$, solution of \eqref{eq_gen_prob_hom}  with $\lambda=\lambda_t$.
    \item Check if $\sigma_t:=f_{y^\delta,b}(\Phi\mu_{\lambda_t})\le \sigma_{\text{target}}(\delta)$,
    where the target value $\sigma_{\text{target}}(\delta)$ depends on the noise level $\delta>0$.
    \item If the condition above is not met, decrease  $\lambda$: $\lambda_{t+1}<\lambda_t$.
\end{itemize}

In the following we discuss the choice of the starting value for the sequence of $\lambda$'s, its updating rule  and the choice of a suitable value $\sigma_{\text{target}}$ for the considered noise scenario.
In Table \ref{tab_hom_gauss_poiss}, we outline the different choices made for  Alg.\ref{alg_homotopy} to both \eqref{eq:blasso} and \eqref{poissonblasso}.
\begin{table}[t]
    \centering
    \begin{tabular}{r|c|c|}
        \cmidrule[\heavyrulewidth]{2-3}
        & \multicolumn{1}{c|}{\eqref{eq:blasso}} & \multicolumn{1}{c|}{\eqref{poissonblasso}} \\
        \midrule
        $\lambda_1$ \, := & $\gamma\|\Phi^*y^\delta\|_\infty$ & $\gamma\left\|\left(\Phi^\ast\big(\frac{y-b}{b}\big)\right)_+\right\|_\infty$ \\
        $\sigma_{t}$ \, := & $\frac{1}{2}\|\Phi\hat{\mu}_t-y^\delta\|^2$ & $\dkl(\Phi\hat{\mu}_t,y^\delta)$ \\
        $\sigma_{\text{target}}(\delta)$ \, := & $\frac{1}{2}\|y-y^\delta\|^2=\frac{\delta^2}{2}$ & $\dkl(y,y^\delta)$ \\
        \bottomrule
    \end{tabular}
    \caption{Homotopy algorithmic choices for problems \eqref{eq:blasso} and \eqref{poissonblasso}.}
    \label{tab_hom_gauss_poiss}
\end{table}

\paragraph{Starting value}
The element $\eta(\lambda,\mu)$ is crucial for the definition of Algorithm \ref{alg_homotopy} as it allows to define a good starting value $\lambda_1$. We propose to initialise as follows: 
\begin{align}
\lambda_1&:=\gamma\left\|\eta(1,\hat{\mu}_0)\right\|_\infty=\gamma \left\|\Tilde{\eta}(\hat{\mu}_0)\right\|_\infty, \qquad \gamma\in(0,1)
    \label{eq_lambda1_hom} 
\end{align}
where $\hat{\mu}_0\in\Mx$ is the initialisation of the solution and $\eta$ is the dual certificate \eqref{eq_gen_dual_cert}. The parameter $\gamma\in(0,1)$ is a relaxation parameter usually chosen close to 1.
The choice \eqref{eq_lambda1_hom} is motivated by optimality conditions \eqref{eq_opt_cond_gen}. If one takes at the first iteration $\lambda_1\ge\left\|\eta(1,\hat{\mu}_0)\right\|_\infty$, then $\hat{\mu}_0$ is an optimal solution for \eqref{eq_gen_prob_hom} with $\lambda=\lambda_1$, since
\begin{equation}    \left\|\eta(\lambda_1,\hat{\mu}_0)\right\|_\infty = \left\|\frac{1}{\lambda_1}\Tilde{\eta}(\hat{\mu}_0)\right\|_\infty=\frac{1}{\lambda_1}\left\|\eta(1,\hat{\mu}_0)\right\|_\infty\le 1 \quad \iff \quad \lambda_1\ge\left\|\eta(1,\hat{\mu}_0)\right\|_\infty.
\end{equation}
In this case, the algorithm does not improve upon the initialisation $\hat{\mu}_0$ since it does not perform any iteration. 
On the contrary, choosing $\lambda_1<\left\|\eta(1,\hat{\mu}_0)\right\|_\infty$ ensures that the initial measure $\hat{\mu}_0$ is updated since the dual certificate computed with respect to the initialisation $\hat{\mu}_0$ and $\lambda_1>0$ is such that 
\eqref{eq_lambda1_hom} has supremum norm that satisfies
\begin{align}
   \left\|\eta(\lambda_1,\hat{\mu}_0)\right\|_\infty &=\frac{1}{\lambda_1}\left\|\eta(1,\hat{\mu}_0)\right\|_\infty=\frac{1}{\gamma\left\|\eta(1,\hat{\mu}_0)\right\|_\infty}\left\|\eta(1,\hat{\mu}_0)\right\|_\infty=\frac{1}{\gamma}>1. 
\end{align}

\paragraph{Updating rule}

The updating rule \eqref{eq_alg_hom_lambdat} for $\lambda$ together with the choice of a strictly positive parameter $c>0$ ensures that the measure $\hat{\mu}_t$, which is used to initialise $(\mathcal{P}_{\lambda_{t+1}})$ as $\mu_{t+1}^{[0]}=\hat{\mu}_{t}$, is not  an optimal solution for the problem. Indeed, the dual certificate computed in correspondence with $\lambda_{t+1}$ and  $\hat{\mu}_t$ reads 
\begin{equation}
\eta(\lambda_{t+1},\hat{\mu}_t)=\frac{\lambda_t}{\lambda_{t+1}}\eta(\lambda_{t},\hat{\mu}_t)=\frac{1+c}{\|\eta(\lambda_{t},\hat{\mu}_t)\|_\infty}\eta(\lambda_{t},\hat{\mu}_t)
    \; \Longrightarrow \; \|\eta(\lambda_{t+1},\hat{\mu}_t)\|_\infty=1+c >1. 
\end{equation}
Thus, $\mu_{t+1}^{[0]}=\hat{\mu}_{t}$ does not satisfy \eqref{eq_opt_cond_gen} for \eqref{eq_gen_prob_hom} with $\lambda=\lambda_{t+1}$, so the homotopy step $t+1$ computes a new candidate solution $\hat{\mu}_{t+1}$. This is indeed consistent with the choice of rejecting $\hat{\mu}_t$ at the previous step $t$.

\subsection{Descent property}

In this section, we show that the homotopy algorithm (Alg.\ref{alg_homotopy})  produces a strictly decreasing sequence of residual errors $(\sigma_t)_{t}$. This properties gives insight on the good convergence of the algorithm. 

\begin{prop}
\label{prop_descent_hom}
    If the minimisation problem \eqref{eq_gen_prob_hom} admits a unique solution, the homotopy algorithm (Alg.\ref{alg_homotopy}) for the resolution of \eqref{eq_gen_prob_hom} produces a strictly decreasing sequence of residual distances $(\sigma_t)_{t\in\N}$.
\end{prop}
\begin{proof}
    Let $\lambda_{t+1}<\lambda_t$, which is true by construction \eqref{eq_alg_hom_lambdat}. We want to show that $\sigma_{t+1}<\sigma_t$.
    Let now $\hat\mu_t$ and $\hat\mu_{t+1}$ be solutions of \eqref{eq_gen_prob_hom} in correspondence with $\lambda_t$ and $\lambda_{t+1}$, respectively. We can rewrite \eqref{eq_gen_prob_hom} as
    \begin{equation}
        \argmin_{\mu\in\Mx} T_{f_{y^\delta,b},\lambda}(\mu) \quad \text{with} \quad  T_{f_{y^\delta,b},\lambda}(\mu):=f_{y^\delta,b}(\Phi\mu)+\lambda|\mu|(\Omega)+\alpha \iota_{\Mxpos}(\mu)
        \label{previous}
    \end{equation}
    and observe that, for $\hat\mu\in\Mxpos$ solution of \eqref{previous},
    \begin{equation}
        T_{f_{y^\delta,b},\lambda}(\hat\mu)=f_{y^\delta,b}(\Phi\hat\mu)+\lambda|\hat\mu|(\Omega).
    \end{equation}
    Note that for any $t\in\N$, the following holds true:
    \begin{align}
        \label{Prop1}
        & \sigma_{t+1}+\lambda_{t+1}|\hat\mu_{t+1}|(\Omega)=T_{f_{y^\delta,b},\lambda_{t+1}}(\hat\mu_{t+1})<T_{f_{y^\delta,b},\lambda_{t+1}}(\hat\mu_{t})=\sigma_{t}+\lambda_{t+1}|\hat\mu_{t}|(\Omega)\\ 
        & \sigma_{t}+\lambda_{t}|\hat\mu_{t}|(\Omega)=T_{f_{y^\delta,b},\lambda_{t}}(\hat\mu_{t})<T_{f_{y^\delta,b},\lambda_{t}}(\hat\mu_{t+1})=\sigma_{t+1}+\lambda_{t}|\hat\mu_{t+1}|(\Omega) \label{Prop2},
    \end{align}
    by optimality of $\hat\mu_t$ and $\hat\mu_{t+1}$ for $T_{f_{y^\delta,b},\lambda_{t}}$ and $T_{f_{y^\delta,b},\lambda_{t+1}}$ respectively.  Rewriting now \eqref{Prop2} as
    \begin{equation}
        \sigma_t+\lambda_{t+1}|\hat\mu_t|(\Omega)+(\lambda_t-\lambda_{t+1})|\hat\mu_t|(\Omega)<\sigma_{t+1}+\lambda_{t+1}|\hat\mu_{t+1}|(\Omega)+(\lambda_t-\lambda_{t+1})|\hat\mu_{t+1}|(\Omega)
    \end{equation}
    yields
    \begin{equation}
        (\lambda_t-\lambda_{t+1})\big(|\hat\mu_{t+1}|(\Omega)-|\hat\mu_{t}|(\Omega)\big)>[\sigma_t+\lambda_{t+1}|\hat\mu_t|(\Omega)]-[\sigma_{t+1}+\lambda_{t+1}|\hat\mu_{t+1}|(\Omega)]>0 \quad \text{by \eqref{Prop1}}.
    \end{equation}
    Hence, since $\lambda_t-\lambda_{t+1}>0$ by hypothesis, we have $|\hat\mu_{t+1}|(\Omega)-|\hat\mu_{t}|(\Omega)>0$, that is $|\hat\mu_{t}|(\Omega)<|\hat\mu_{t+1}|(\Omega)$.  We can thus deduce from \eqref{Prop1} that
    \begin{equation}
        \sigma_t-\sigma_{t+1}>\lambda_{t+1}\big(|\hat\mu_{t+1}|(\Omega)-|\hat\mu_{t}|(\Omega)\big)>0
    \end{equation}
    since $|\hat\mu_{t+1}|(\Omega)-|\hat\mu_{t}|(\Omega)>0$ and $\lambda_{t+1}>0$.
    Thus, we obtain $\sigma_{t+1}<\sigma_t$, which concludes the proof.
\end{proof}
Observe that, under the hypothesis of injectivity of the forward operator $\Phi$, the latter result is valid both for \eqref{eq:blasso} \cite{bredies2013inverse} and for \eqref{poissonblasso} (Proposition \ref{poissonblasso_uniqueness}).

\begin{oss}
    The proof of Proposition \ref{prop_descent_hom} holds for any strictly decreasing updating rule.  
    However, using an updating rule different from \eqref{eq_alg_hom_lambdat} there is no guarantee that at iteration $t+1$ of Algorithm \ref{alg_homotopy} the measure $\mu_{t+1}^{[0]}=\hat{\mu}_{t}$ is not already optimal for $\mathcal{P}(\lambda_{t+1})$, thus requiring an immediate update of $\lambda$. 
\end{oss}

\section{Numerical tests} \label{sec:results}

In this section, we report numerical results to \eqref{poissonblasso} computed by means of the  Sliding Frank-Wolfe (SFW) Algorithm \ref{alg_sfw} with homotopy (Algorithm \ref{alg_homotopy}) for several off-the-grid sparse deconvolution problems in simulated 1D/2D and real 3D fluorescence microscopy data. 

\subsection{Simulated 1D experiments}

\begin{figure}[t]
    \centering
    \includegraphics[width=0.8\textwidth]{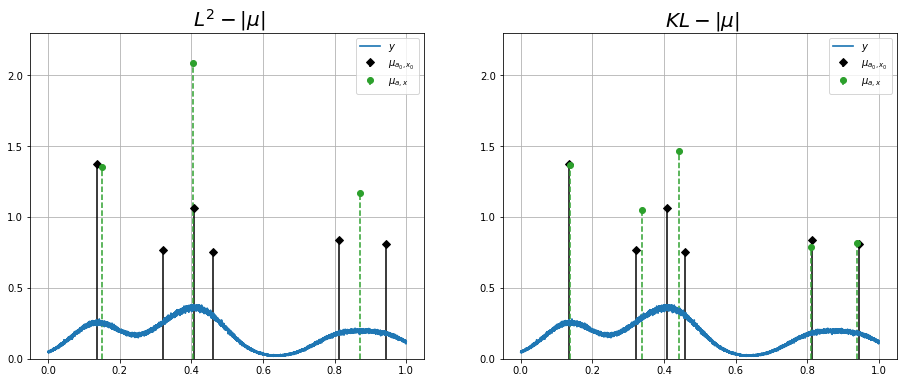}
    \caption{1D comparison between Gaussian (left) and Poisson (right) models. In black: ground truth spikes. In green: reconstructed spikes. For both models, $\lambda=8.82$.}
    \label{fig_1D_rec}
\end{figure}

The aim of this first set of experiments on simulated 1D blurred signals corrupted with Poisson noise is to validate our model \eqref{poissonblasso} and to compare its performance with BLASSO \eqref{eq:blasso}. In a discrete setting, several works (e.g., \cite{Bertero2018,Harmany2012,Calatroni2017} and \cite{Johnson24phase} for examples in microscopy) have proposed both analytical and numerical approaches for precisely modelling signal-dependent Poisson noise using discrete version of the KL divergence \eqref{eq_kl}. Generally speaking, such choices improve performance (in terms of, e.g., localisation/reconstruction quality, especially in low-photon count regimes) in comparison to simpler $L^2$ Gaussian models, but only slightly. This is due to the biases introduced by the handcrafted regularisation employed, such as, e.g., $\ell_1$ or TV-type. Aiming at reconstructing weighted sums of Diracs, for which the regulariser \eqref{eq:TV_measure} is tailored, we wonder in the following experiments whether in such infinite-dimensional setting the improvement is more evident.

For that, we test 10 different ground truth signals with 6 randomly located spikes in the 1D domain $\Omega=[0,1]$. For each ground truth signal, the position of each spike is sampled from a uniform distribution over $\Omega$, as well as their amplitudes from a uniform distribution over $[1-d,1+d]$ with $d=0.4$. The corresponding acquisitions are blurred by a Gaussian 1D PSF with $\sigma=0.07$, a spatially constant background $b=0.01$ is considered, and then several Poisson noise realisations are generated as acquired data. In Figure \ref{fig_1D_rec}, one simulated ground truth signal $\mu_\text{gt}$ is shown (black spikes) together the corresponding Poisson noisy and blurred data (blue signal). All 1D signals are then reconstructed by using both the Poisson  \eqref{poissonblasso} and the Gaussian noise model BLASSO \eqref{eq:blasso} with $\lambda\in(0,10]$ using Algorithm \ref{alg_sfw}. 
Figure \ref{fig_1D_rec} shows an example between the two reconstructions (BLASSO, left, Poisson model, right)  $\mu_\text{rec}$ (green spikes) for $\lambda=8.82$. For such illustrative example, the Poisson model provides a better estimate (in terms both of amplitudes and localisation) than the Gaussian model.  

To assess such observation over different choices of regularisation parameters (to which the models could be sensitive), we then performed a statistical test comparing localisation/reconstruction performance for all the generated signals.
To evaluate the goodness of the reconstructions, we consider the Jaccard index defined in terms of the number of True Positive (TP), False Positive (FP) and False Negative (FN) spikes as follows
\begin{equation}
\text{Jac}_\delta(\mu_\text{gt},\mu_\text{rec})=\frac{\#\text{TP}}{\#\text{TP}+\#\text{FP}+\#\text{FN}}\in[0,1]
\end{equation}
with tolerance radius $\delta>0$. 
TP are reconstructed spikes that are at a distance less than $\delta$ from a ground truth spike, while reconstructed spikes that are more than $\delta$ distant from each ground truth spike are denoted by FP. FN are spikes in the ground truth which have not been associated to any TP.

\begin{figure}[h]
    \centering
    \begin{tabular}{c}
        \includegraphics[width=\textwidth]
{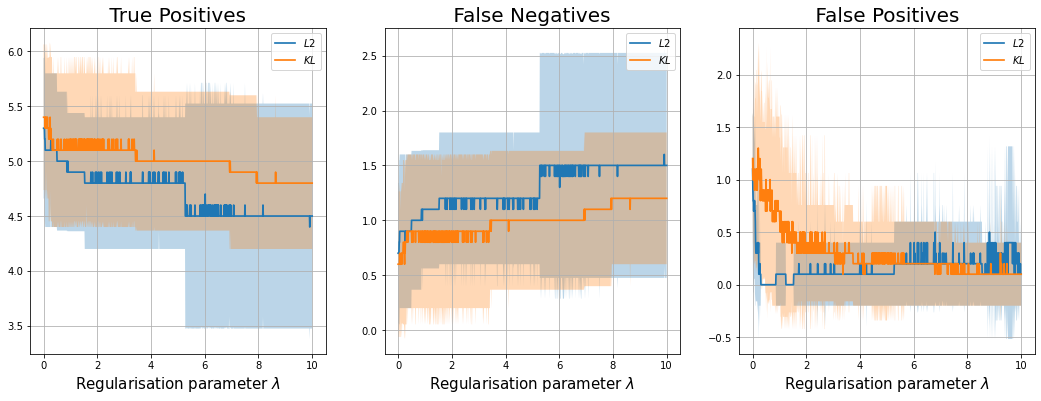}\\
(a) TP, FN and TP\\
\includegraphics[width=\textwidth]
{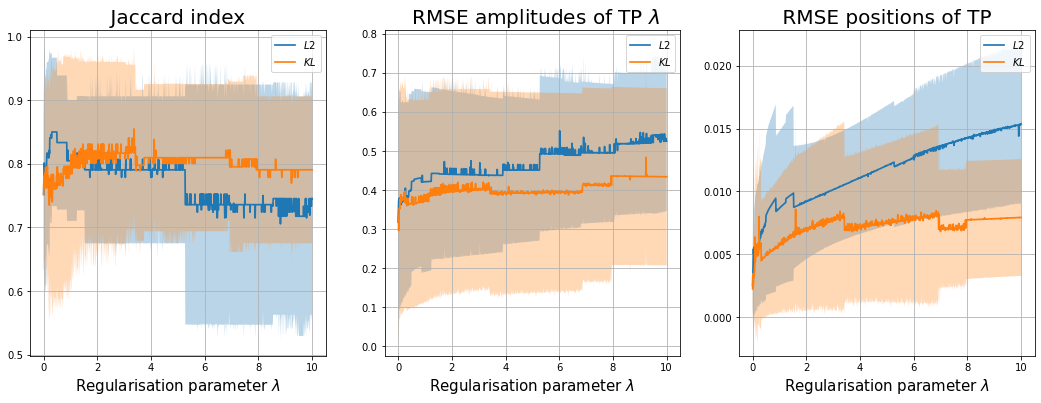}\\
(b) Jaccard index and RMSE of amplitudes and positions
    \end{tabular}
    \caption{Mean values over 100 different randomly generated ground truth signals with 6 spikes and their corresponding reconstructions. Shaded area corresponds to standard deviation. Maximum number of iterations of SFW: $2N_{\text{spikes}}$. Tolerance radius for computation of the Jaccard Index is $\delta=0.05$.}
    \label{fig_1D_jacc}
\end{figure}

\begin{table}[t]
    \centering
\begin{tabular}{l|ll}
        \toprule
        Parameters & $L^2-|\mu|$ & $\dkl-|\mu|$ \\
        \midrule
        Max. number of homotopy iterations & $2 N_{\text{molecules}}$ & $2 N_{\text{molecules}}$ \\
        Max. number of inner SFW iterations & 1 & 1 \\
        Homotopy parameter $c$ & 15 & 40 \\
        Homotopy parameter $\gamma$ & 0.9 & 0.9\\
        Choice of $\sigma_{\text{target}}$ & $1.5 \cdot \frac{1}{2}\|\Phi\mu_{\text{gt}}+b-y\|^2$ & $1.5 \cdot \kl(\Phi\mu_{\text{gt}}+b,y)$ \\
        \bottomrule
    \end{tabular}
    \caption{Parameters used by the homotopy algorithm (Alg.\ref{alg_homotopy}) in the 1D simulated comparison tests.}
    \label{tab_hom_parameters}
\end{table}
\begin{table}[th!]
    \centering
\begin{tabular}{lcc}
        \cmidrule[\heavyrulewidth]{2-3}
        & \multicolumn{1}{c}{$L^2-|\mu|$} & \multicolumn{1}{c}{$\dkl-|\mu|$} \\
        \midrule
        Jaccard index & 0.74  & \textbf{0.76} \\
        Number of TP & 4.50  & \textbf{4.80} \\
        Number of FN & 1.50  & \textbf{1.20} \\
        Number of FP & \textbf{0.10}  & 0.40 \\
        RMSE on amplitudes of TP & 0.41  & 0.44 \\
        RMSE on positions of TP & 0.014 & 0.015 \\
        \midrule
        Final estimated $\mathbf\lambda$ & 6.09 & 40.21 \\
        Number of homotopy iterations & 4.55 & 3.93 \\
        Value of $\sigma_{\text{target}}$ & 4.09 & 77.16 \\
        \bottomrule
    \end{tabular}
    \caption{Homotopy algorithm: comparison between BLASSO and the Poisson off-the-grid modelling. Mean values over 100 different randomly generated ground truths with 6 spikes.}
    \label{tab_hom_confronto}
\end{table}
Similarly as in \cite{Denoyelle2019}, we used also the RMSE of the amplitudes $a$ and positions $x$ of the TP spikes as a different quality metric:
\begin{align}
    & \text{RMSE}_x(\mu_\text{gt},\mu_\text{rec})=\sqrt{\frac{1}{\#\text{TP}}\sum_{i\in \text{TP}}\big((x_\text{rec})_i-(x_{gt})_i\big)^2 } \\ 
    & \text{RMSE}_a(\mu_\text{gt},\mu_\text{rec})=\sqrt{\frac{1}{\#\text{TP}}\sum_{i\in \text{TP}  }\big((a_\text{rec})_i-(a_{gt})_i\big)^2}.
\end{align}
In Figure \ref{fig_1D_jacc}(a), we plot the number of  TP, FN and FP reconstructed by models \eqref{eq:blasso} and \eqref{poissonblasso}  for a large finely discretised range of $\lambda$. In Figure \ref{fig_1D_jacc}(b) the Jaccard index (computed with $\delta=0.05$) and the Root Mean Square Error (RMSE) of amplitudes and positions are also reported.
The proposed Poisson model \eqref{poissonblasso} has a better performance in terms of TP and FN and, overall, in terms of the Jaccard index and RMSE of amplitudes and positions. Only for the number of FP, BLASSO \eqref{eq:blasso} results slightly better than \eqref{poissonblasso} for small values of $\lambda$. This is due to the fact that \eqref{poissonblasso} usually requires more iterations of SFW before reaching convergence. This results in a better estimation of the number of molecules, with TP being closer to the actual number of spikes, which may cause an overestimation of the number of spikes with a consequently higher value of FP. 
\smallskip

Using the same dataset, we then compare the results obtained running the  homotopy algorithm (Alg.\ref{alg_homotopy}) for the automatic selection of the regularisation parameter $\lambda$ for both problems \eqref{eq:blasso} and \eqref{poissonblasso} solved by Algorithm \ref{alg_sfw} as inner solver, with the algorithmic parameters specified in Table \ref{tab_hom_parameters}. 
To reduce the computational burdens, we observed that only one iteration of SFW was enough, as the estimated measures are anyway updated in the next homotopy step. For the same reason, we set the maximum number of homotopy outer iterations to be equal to twice the number of peaks in the ground truth.
As far as $\sigma_{\text{target}}$ is concerned, being in a simulated environment, we  computed the exact value of the residual $\sigma_{\text{exact}}=f_{y^\delta,b}(\Phi\mu_\text{gt})$ and set $\sigma_\text{target}=1.5\cdot\sigma_\text{exact}$, to be compared with the residual $\sigma_t$ at current iteration $t$ of Algorithm \ref{alg_homotopy}.  Since $\sigma_{\text{exact}}$ is unknown in real situations, a possible strategy for its estimation will be discussed in the next section.
In Table \ref{tab_hom_confronto}, we report the values of TP, FN, FP, Jaccard index and RMSE of the reconstructed signals for both models. The final estimated $\lambda$ is also reported in Table \ref{tab_hom_confronto}, together with the number of performed homotopy iterations and of the value $\sigma_{\text{target}}$. Using homotopy, we observe that we retrieve values which are comparable with the best ones obtained using SFW with grid search. This shows the effectiveness of the homotopy strategy. Overall, the algorithm applied to solve \eqref{poissonblasso} yields better results than \eqref{eq:blasso} in the presence of Poisson noise, with a reduction of the number of FN and an improvement of the accuracy in terms of Jaccard index. 

\subsection{Simulated 2D and 3D examples: choice of $\sigma_{\text{target}}$}

To avoid the choice of $\sigma_{\text{target}}$ to depend on the ground truth image, we propose in this section an heuristic strategy to estimate a reasonable value $\sigma_{\text{target}}$ relying on the sole acquisition $y^\delta$ and on the assumption that the signal is sparse. To better illustrate the proposed strategy we consider a 2D example of simulated blurred and noisy microscopy acquisitions on the domain $\Omega=[0,1]^2$. The 2D simulated ground truth has 15 spikes with positions randomly sampled from a uniform distribution over $\Omega$, and amplitudes sampled from a uniform distribution over $[0.5,1.5]$. The corresponding acquisition is blurred by a 2D Gaussian PSF with $\sigma=0.07$ and a spatially constant background $b=0.05$ is considered. Then, a Poisson noise realisation is considered as $y^\delta$, sampled from a Poisson random variable with mean $\Phi\mu_\text{gt}+b$.

Under a suitable sparsity level of the ground truth image, it is safe to assume that its corresponding noisy and blurred acquisition $y^\delta$ presents regions containing only background noise, which we denote by $\Omega_{\text{bg}}\subset \Omega$. In Figure \ref{fig_2D_hom} we show $y^\delta$ and highlight $y^\delta|_{\Omega_{\text{bg}}}$ in transparency, i.e. the acquisition "masked" to the area of background noise in the external square-ring. We propose to estimate the value of $\sigma_{\text{target}}$ as follows
\begin{equation}
    \sigma_{\text{target}} = f_{y^\delta,b}(0)\big|_{\Omega_{\text{bg}}}\frac{|\Omega|}{|\Omega_{\text{bg}}|},
    \label{eq_est_sigma_target}
\end{equation}
where the restriction of the fidelity term to $\mu=0$ is due to the fact we assume the desired image $\mu$ to be null in $\Omega_{\text{bg}}$, i.e. $\mu|_{\Omega_{\text{bg}}}=0$. Note that considering $\sigma_{\text{target}} = f_{y^\delta,b}(0)\big|_{\Omega_{\text{bg}}}$  would be equivalent to assume that the noise is null on $\Omega\setminus\Omega_{\text{bg}}$, which is obviously not true. The formula \eqref{eq_est_sigma_target} is thus adjusted to account for noise not only on $\Omega_{\text{bg}}$ but on the whole domain $\Omega$. Note that it is also possible to approximately estimate the (constant) background $b\in L^2(\Omega)$ in $\Omega_{\text{bg}}$ by taking
\begin{equation}
    b=\frac{1}{|\Omega_{\text{bg}}|}\int_{\Omega_{\text{bg}}}y^\delta(t)\diff t.
    \label{eq_est_background}
\end{equation}

\begin{table}[b]
    \centering
    \begin{tabular}{l|l|l}
    \toprule
        Theoretical value (based on the ground truth) & $\dkl(\Phi\mu_{gt}+b,y^\delta)$ & 842.3\\
        Poisson discrepancy principle \eqref{eq_bertero} \cite{Bertero2010ADP} & $\frac{|\Omega|}{2}$ & 8192\\
         Estimation based only on $y^\delta$ and $\Omega_{\text{bg}}$ \eqref{eq_est_sigma_target} & $\dkl(b,y^\delta_{bg}) \frac{|\Omega|}{|\Omega_{\text{bg}}|}$ & 846.4\\
         \bottomrule
    \end{tabular}
    \caption{Different estimates of $\sigma_{\text{target}}$ in the 2D simulated setting}.
    \label{tab_2D_hom_sigma_target}
\end{table}
\begin{figure}[b]
\centering
\begin{subfigure}[b]{0.4\textwidth}
    \includegraphics[width=\textwidth]{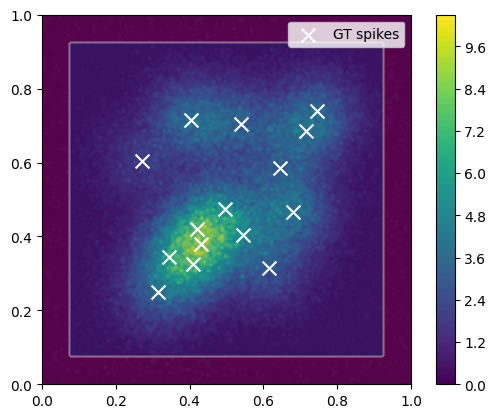}
    \caption{Data}
\label{fig_2D_hom}
\end{subfigure}
\begin{subfigure}[b]{0.4\textwidth}
    \includegraphics[width=\textwidth]{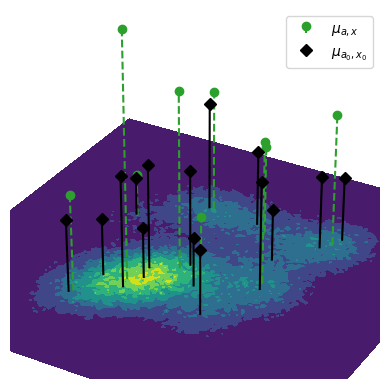}
    \caption{Reconstruction}
\end{subfigure}\\
\begin{subfigure}[b]{\textwidth}
    \includegraphics[width=\textwidth]{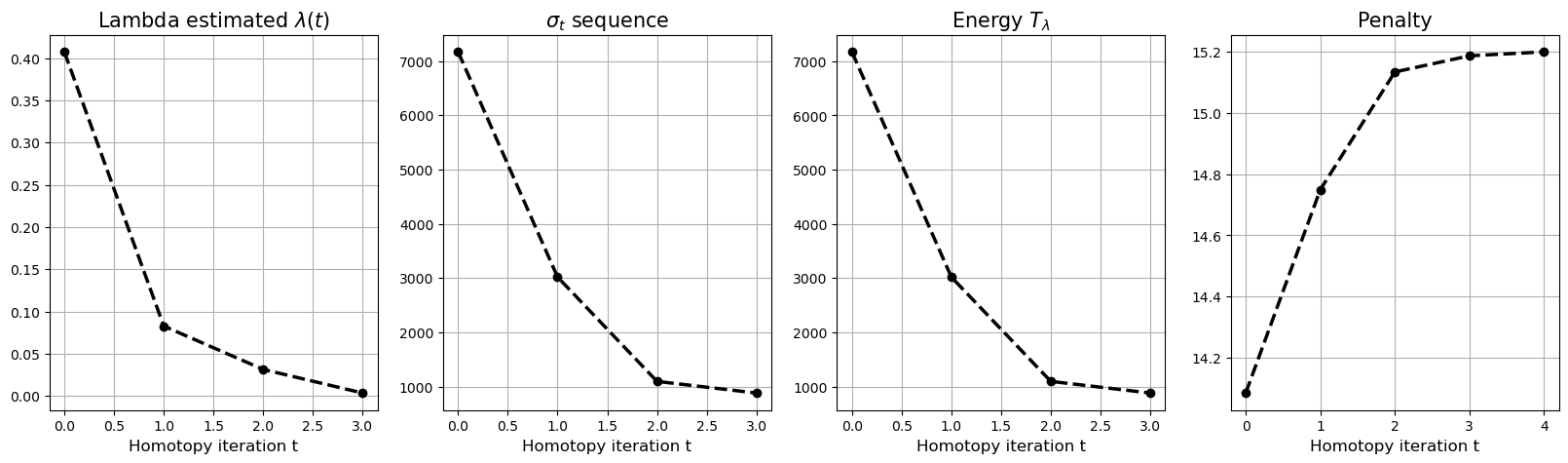}
    \caption{Values of $\lambda_t$, $\sigma_\text{target}$ and cost functional along the homotopy iterations.}
\end{subfigure}
\caption{2D fluorescence microscopy simulated image. (a) 2D sparse image (white crosses denote off-the-grid locations) and its corresponding noisy blurred acquisition $y^\delta$ with $y^\delta|_{\Omega_{\text{bg}}}$ visualised in trasparency. (b) 2D reconstruction (green spikes) obtained with homotopy algorithm (Alg. \ref{alg_homotopy}) compared with the ground truth spikes (black).}
\label{eq_2D_results}
\end{figure}
For the image shown in Figure \ref{fig_2D_hom}, we compare in Table \ref{tab_2D_hom_sigma_target} different choices of $\sigma_{\text{target}}$ for \eqref{poissonblasso} with $f_{y^\delta,b}$ being the Kullback-Leibler fidelity term. In particular, we consider in the second row the estimate proposed in \cite{Bertero2010ADP}, where a discrepancy principle for Poisson data is studied under the following approximation
\begin{equation}
    \dkl(\Phi\mu_{gt}+b,y^\delta)\approx \frac{|\Omega|}{2}.
    \label{eq_bertero}
\end{equation}
This value is obtained by computing the expected value for Kullback-Leibler fidelity and by approximating it with a first order Taylor expansion. As observed also in \cite{Bertero2010ADP}, the estimate \eqref{eq_bertero} might not be optimal and, indeed, one should consider 
\begin{equation} \label{eq:estimate_bertero}
    \dkl(\Phi\mu_{gt}+b,y^\delta)\approx \frac{1-\epsilon}{2}|\Omega|,
\end{equation}
where $\epsilon\in\R$ is small. When $\Omega$ is big, \eqref{eq:estimate_bertero} might lead to bad estimates even if $\epsilon$ is very small. As shown in Table \ref{tab_2D_hom_sigma_target}, the approximation \eqref{eq_bertero} does not lead indeed  to an accurate estimation of $\sigma_{\text{target}}$. On the contrary, the estimation \eqref{eq_est_sigma_target} proposed (last row), is close to the real value (first row) which is known given the simulated setting. 
We remark that in recent work \cite{Bevilacqua_2023_masked} similar masking strategies are used to define parameter selection strategies for variational noise in Poisson scenarios, with a detailed description of the  modifications arising to the underlying statistical laws when performing such choice. The application of analogous strategies to the problem considered is an interesting venue for future work.

By using the homotopy algorithm (Alg.\ref{alg_homotopy}) with SFW (Alg.~\ref{alg_sfw}) \footnote{Max. number of (outer) homotopy it. 20, max. number of (inner) SFW it. 1, $c=30$, $\gamma=0.9$.} for the reconstruction of the data in Figure \ref{fig_2D_hom} with background and $\sigma_{\text{target}}$ estimated by \eqref{eq_est_background} and \eqref{eq_est_sigma_target}, respectively, we obtain the results shown in Figure \ref{eq_2D_results}.

\begin{table}[t]
    \centering
    \begin{tabular}{l|l|l}
    \toprule
        Theoretical value (based on the ground truth) & $\dkl(\Phi\mu_{gt}+b,y^\delta)$ & 326.8\\
        Poisson discrepancy principle \eqref{eq_bertero} \cite{Bertero2010ADP} & $\frac{|\Omega|}{2}$ & 6400\\
         Estimation based only on $y^\delta$ and $\Omega_{\text{bg}}$ \eqref{eq_est_sigma_target} & $\dkl(b,y^\delta_{bg}) \frac{|\Omega|}{|\Omega_{\text{bg}}|}$ & 375.7\\
         \bottomrule
    \end{tabular}
    \caption{Different estimates of $\sigma_{\text{target}}$ in the 3D simulated setting for \eqref{poissonblasso}.}
    \label{tab_3D_hom_sigma_target}
\end{table}
To conclude this section, we present a test of the homotopy algorithm (Alg.\ref{alg_homotopy}) with the proposed choice of $\sigma_\text{target}$ on an exemplar 3D simulated setting. We consider $\Omega=[-1300,1300]\times[-1300,1300]\times[-1000,1000]$ as domain, and a simulated ground truth measure with 7 spikes. Their positions are sampled from a uniform distribution over $\Omega$ and their amplitudes are sampled uniformly from $[1-d,1+d]$ with $d=0.4$. We show the ground truth spikes' position with white crosses in Figure \ref{fig:sim3d}, projected on the 3 planes $xz, yz, yx$. To simulate the corresponding blurred acquisition  we consider a 3D Gaussian PSF 
\begin{equation}
    \varphi\big(x,y,z\big)=\frac{1}{\sqrt{(2 \pi)^3} \sigma_x\sigma_y\sigma_z}\exp \left[-\frac{x^2}{2 \sigma_x^2}\right]\exp \left[-\frac{y^2}{2 \sigma_y^2}\right]\exp \left[-\frac{z^2}{2 \sigma_z^2}\right]
    \label{eq:3dgausspsf}
\end{equation}
 \begin{figure}[t]
    \centering
    \begin{subfigure}{\textwidth}
        \includegraphics[width=\textwidth]{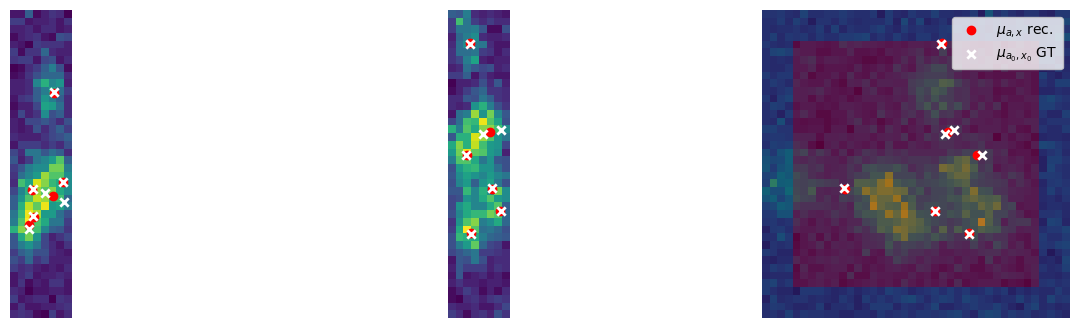}
        \caption{In red, the \eqref{poissonblasso} reconstruction is shown over the noisy acquisition $y^\delta$. The shaded area in the top view (plane $yx$, III column) highlights the area $\Omega_\text{bg}$ and $y^\delta_{\Omega_\text{bg}}$, necessaries for the estimates \eqref{eq_est_sigma_target} and \eqref{eq_est_background}.}
        \label{fig:sim3d_KL}
    \end{subfigure}\\
    \begin{subfigure}{\textwidth}
        \includegraphics[width=\textwidth]{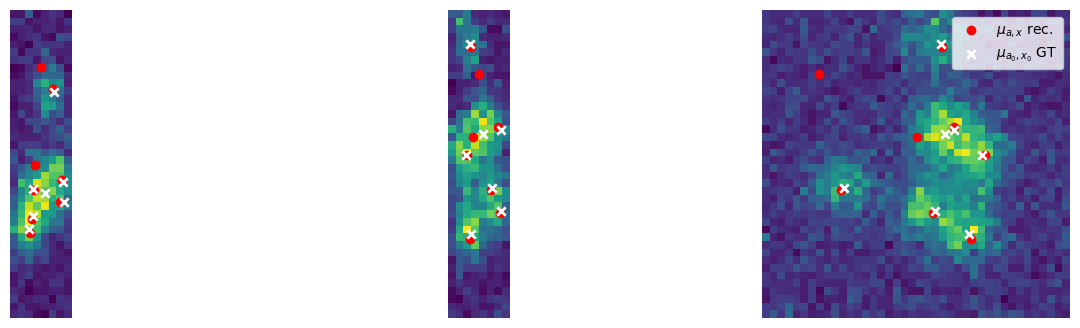}
        \caption{In red, the \eqref{eq:blasso} reconstruction is shown over the noisy acquisition $y^\delta$.}
        \label{fig:sim3d_L2}
    \end{subfigure}
    \caption{3D fluorescence microscopy simulated image. Both in (a) and in (b), the 3D sparse ground truth volume is reported with white crosses (that denote the off-the-grif locations on $\Omega$). The corresponding noisy and blurred acquisition $y^\delta$ is shown with maximum intensity projections over the $xz, yz, yx$ planes. With red dots, the off-the-grid positions of the reconstructions are shown.}
    \label{fig:sim3d}
\end{figure}
\begin{table}[t]
    \centering
\begin{tabular}{l|ll}
        \toprule
        Parameters & $L^2-|\mu|$ & $\dkl-|\mu|$ \\
        \midrule
        Max. number of homotopy iterations & $N_{\text{molecules}}+1$ & $N_{\text{molecules}}+1$ \\
        Max. number of inner SFW iterations & 2 & 2 \\
        Homotopy parameter $c$ & 5 & 20 \\
        Homotopy parameter $\gamma$ & 0.9 & 0.9\\
        Choice of $\sigma_{\text{target}}$ &  $1.1\cdot\frac{1}{2}\|\Phi\mu_\text{gt}+b-y^\delta\|^2$ & $\dkl(b,y^\delta_{bg}) \frac{|\Omega|}{|\Omega_{\text{bg}}|}$\\
         &  (\textit{exact} value) & (estimate \eqref{eq_est_sigma_target})\\
        \bottomrule
    \end{tabular}
    \caption{Parameters used by the homotopy algorithm (Alg.\ref{alg_homotopy}) in the 3D simulated comparison tests.}
    \label{tab_hom_parameters_3d}
\end{table}
with $\sigma_x=\sigma_y=200$ and $\sigma_z=400$. We set the voxel size of 65nm in $yz$ and 250nm in $z$ and added a spatially constant background $b=0.5$. The simulated blurred and noisy acquisition $y^\delta$ is shown in Figure \ref{fig:sim3d} as maximum intensity projections on the planes $xz, yz, yx$; Poisson noise is considered. 

Then, we compute the \eqref{poissonblasso} and \eqref{eq:blasso} reconstructions of the considered volume using the homotopy algorithm (Alg.\ref{alg_homotopy}) with SFW (Alg.\ref{alg_sfw}) as an inner solver, whose parameters are specified in Table \ref{tab_hom_parameters_3d}. In this test, we decided to use the estimates given by \eqref{eq_est_background} for the background and by \eqref{eq_est_sigma_target} for the value of $\sigma_\text{target}$, respectively. The region $\Omega_\text{bg}$ is highlighted in Figure \ref{fig:sim3d_KL} in transparency. We observe in Table \ref{tab_3D_hom_sigma_target} the effectiveness of the estimate given by \eqref{eq_est_sigma_target} for the choice of $\sigma_\text{target}$ in the case of Poisson noise. We remark that also the estimate of the background value given by \eqref{eq_est_background} is good: indeed, the estimated value is 0.515 (with an exact value of 0.5). For \eqref{eq:blasso}, we considered the estimated value for the background coupled with the \textit{exact} value for $\sigma_\text{target}$, computed knowing the ground truth, that is $\sigma_\text{target}=1.1\cdot\frac{1}{2}\|\Phi\mu_\text{gt}+b-y^\delta\|^2$, since the estimated value was not very accurate. This may be due to the fact that $y^\delta$ is affected by Poisson noise, which acts differently in the background and in the foreground, and, hence, using a valid rule in case of Gaussian noise (which is not signal dependent) might lead to inaccurate estimates.

The final results are shown in Figure \ref{fig:sim3d}: it results evident that the Poisson model \eqref{poissonblasso} performs better in this scenario and that the homotopy algorithm (Alg.\ref{alg_homotopy}) with the proposed estimates for the background value \eqref{eq_est_background} and for $\sigma_\text{target}$ \eqref{eq_est_sigma_target} is particularly effective.

\subsection{Real 3D dataset}
\begin{figure}[b]
    \centering
    \includegraphics[width=\textwidth]{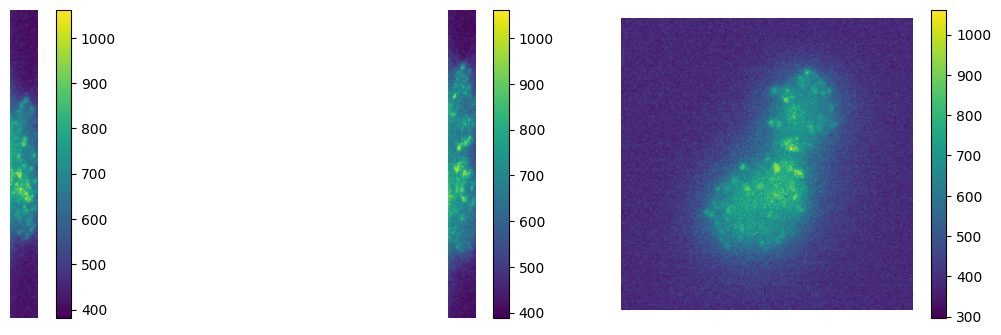} 
    \caption{ERES 3D real data (I and II columns: lateral views, III column: top view).}
    \label{fig_3d}
\end{figure}
We now consider real 3D fluorescence microscopy blurred and noisy volume data, acquired using a TIRF microscope.  The image was taken on a Nikon Ti2 with a 100x/1.49 Oil objective (TIRF), by Alejandro Melero at the MRC-LMB and was used in \cite{Navarro2020_eres3d}. 
It is an acquisition of yeasts expressing fluorescent proteins (*SEC16-sfGFP* and
a *SEC24-sfGFP*) localised at the Endoplasmic Reticulum exit sites (ERES). 
The acquired volume $y^\delta$ is shown in Figure \ref{fig_3d} with maximum intensities projections over the $xz$, $yz$, $yx$ planes. The 3D volume blurred and noisy acquisition has $190\times 190\times 17$ voxels with voxel size of 65nm in $yx$ and 250nm in $z$. Signal dependency of the noise is observed.

To reconstruct a sparse volume from the  3D acquisition $y^\delta$, we use the homotopy algorithm (Alg. \ref{alg_homotopy}) with an accelerated version of the SFW, called Boosted SFW proposed in \cite{Courbot2021}, as an inner solver to minimise \eqref{poissonblasso}. Boosted SFW reduces computational costs by limiting the number of sliding steps.
\begin{table}[t]
    \centering
    \begin{tabular}{l |c| }
    \toprule
  Parameters  & $\eqref{poissonblasso}$ \\
  \midrule
 Max. number of homotopy iterations & 10\\
 Max. number of inner SFW iterations & 50 \\
 Homotopy parameter $c$ &  $0.5$\\
 $\sigma_{\text{target}}$ given by \eqref{eq_est_sigma_target}& 1102067.75\\
 $\sigma_{\text{target}}$ given by \eqref{eq_bertero} & 306850\\
 Constant background estimate \eqref{eq_est_background} & b = 337.77\\
 \bottomrule
\end{tabular}
\caption{Parameters used in Algorithm \ref{alg_homotopy} for the reconstruction of the 3D volume}
\label{tab_3d_hom}
\end{table}

We consider a 3D convolution kernel \eqref{eq_conv_op} estimated as a 3D Gaussian PSF \eqref{eq:3dgausspsf}.
The standard deviation $\sigma_x,\sigma_y$ of the 3D PSF can be estimated from the Full Width Half Maximum, which is given by 
$ \text{FWHM}=0.61 \lambda_\text{wavelength}/\text{NA}$,
where $\lambda_\text{wavelength}$ is the emission wavelength of the fluorescent proteins and NA is the numerical aperture of the microscope. Note that if the FWHM is known, then it is possible to retrieve information about the variance parameters of the PSF since $\text{FWHM}=2.355 \cdot \sigma_x$ and $\text{FWHM}=2.355 \cdot \sigma_y$. For the standard deviation in the $z$-axis, we assume $\sigma_z=2\cdot\sigma_x$. Since the value of the numerical aperture is known, $\text{NA}=1.49$, and $\lambda_\text{wavelength}=508$nm for the green fluorescent proteins under test, we obtain a PSF estimation with $\sigma_x=\sigma_y=89$nm and $\sigma_z=178$nm, which appears to be a good approximation of the underlying blur model. 

\begin{figure}[b]
    \centering
    \includegraphics[width=\textwidth]{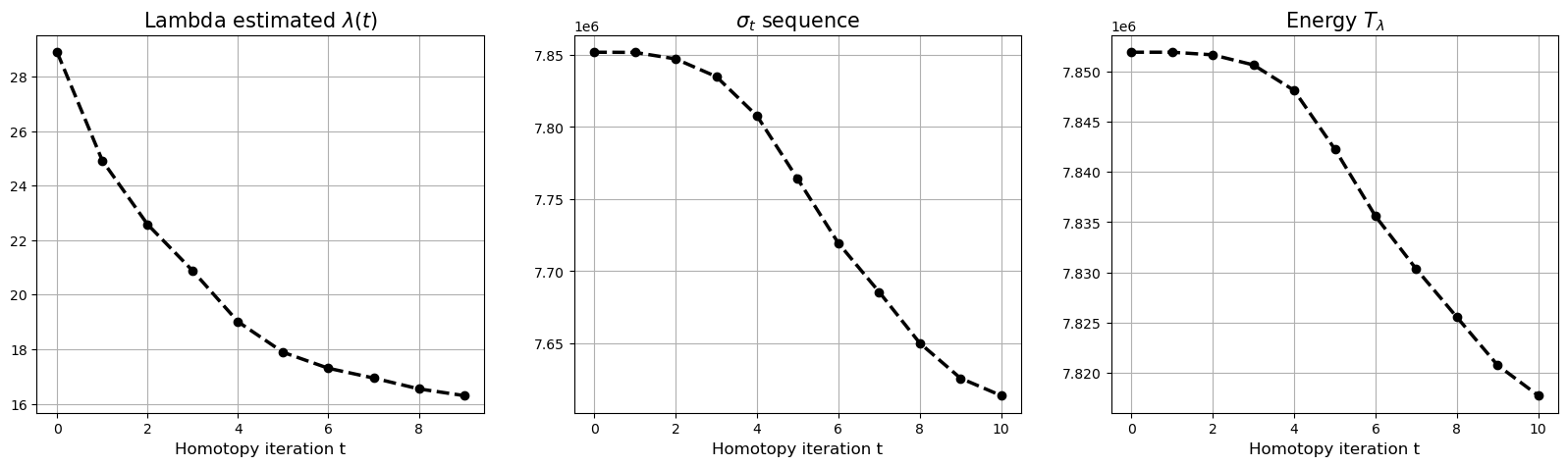}
    \caption{ERES 3D data. Values of $\lambda_t$, $\sigma_\text{target}$ and cost functional along the homotopy iterations. }
    \label{fig_3d_sigma}
\end{figure}

The parameters used to run the homotopy algorithm (Alg.\ref{alg_homotopy}) with BSFW as inner solver are reported in Table \ref{tab_3d_hom}. Note that as far as the estimate of $\sigma_{\text{target}}$ is concerned, the formulas \eqref{eq_bertero} and \eqref{eq_est_sigma_target} give very different results, so that  the estimate given by \eqref{eq_est_sigma_target} was considered as more accurate as shown in the previous section. The background is estimated by \eqref{eq_est_background} as $b=337.77$. We computed this reconstruction using  50 homotopy iterations fixing to 10 the maximum number of inner loops of the BSFW algorithm, using Google Colab CPUs for about 10 hours. The reconstructed volume $\mu_\text{rec}$ counted 274 spikes.
While probably underestimating the exact number of spikes, this first result is promising, since the use of Algorithm \ref{alg_homotopy} yields very precise localisation of spikes automatically, with no need of estimating the regularisation parameter, and no a-priori information about the data. A better visualisation of the reconstruction under different views is given in Figure \ref{fig_eres_3dscatter} using the visualisation codes used in \cite{Denoyelle2019} and provided by the authors at their GitHub page\footnote{\url{https://github.com/qdenoyelle/sfw4blasso}}.
\begin{figure}
    \centering
    \begin{tabular}{ccc}
      \includegraphics[width=0.3\textwidth]{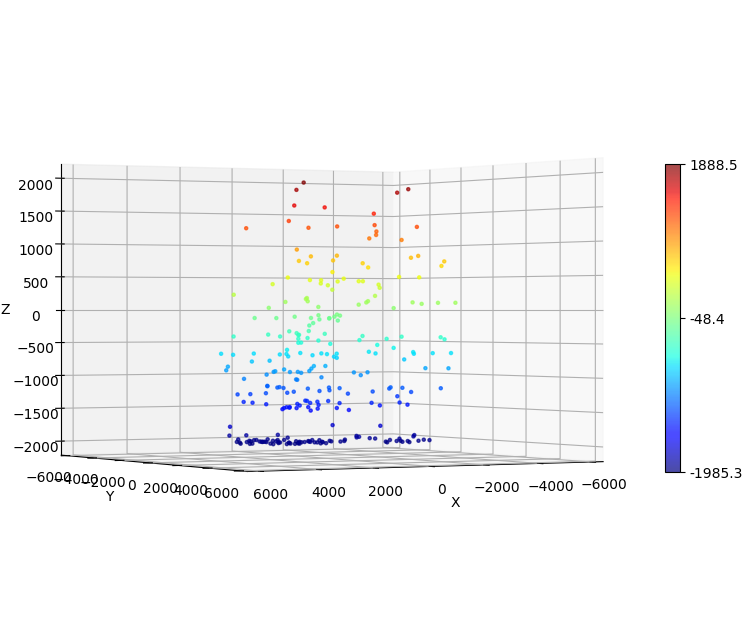}   & \includegraphics[width=0.3\textwidth]{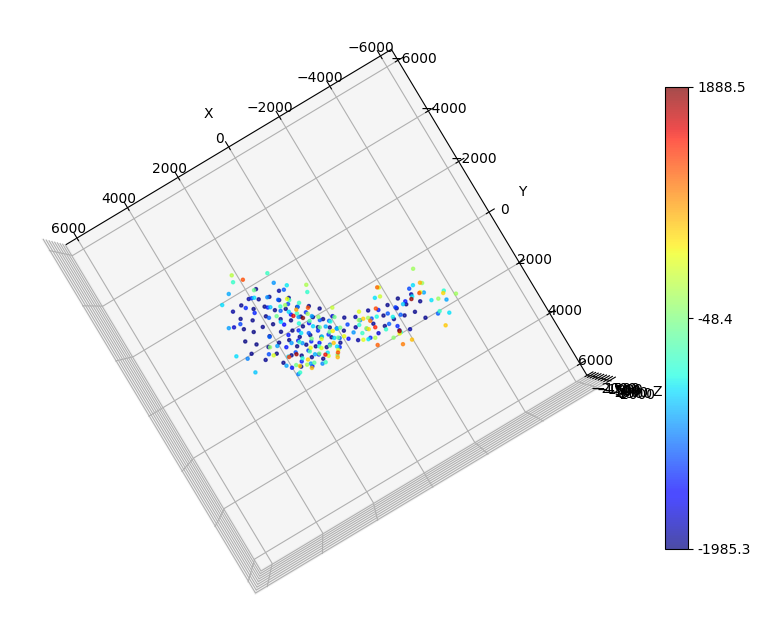}  & \includegraphics[width=0.3\textwidth]{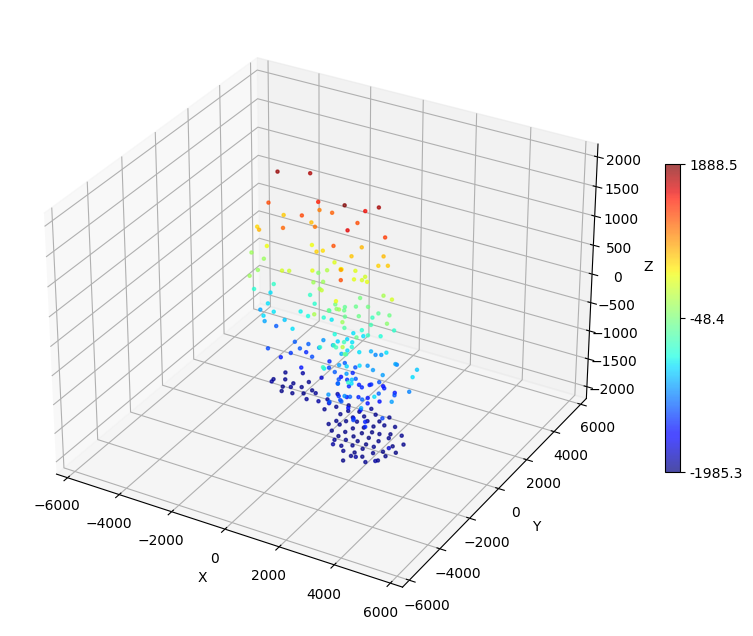}
    \end{tabular}
    \caption{Sparse reconstruction of the 3D real ERES data}
    \label{fig_eres_3dscatter}
\end{figure}

\section{Conclusions}

In this work, we considered sparse inverse problems in the off-the-grid setting of the space of Radon measures $\Mx$ under the assumption of signal-dependent Poisson noise in the measurement. Such choice is motivated by fluorescence microscopy applications, where the noise observed is modelled as Poisson to account for photon-counting processes. First, we designed a variational approach where Total Variation regularisation is coupled with a Kullback-Leibler fidelity term and a non-negativity constraint and derived analytically the optimality and extremality conditions.  
 Then, we considered the Sliding Frank-Wolfe algorithm as a numerical solver and discussed how to select a good regularisation parameter by means of an algorithmic homotopy strategy. 
 Finally, we presented several numerical experiments on simulated 1D/2D/3D data to validate the theoretical findings as well as to to compare the proposed approach with the Gaussian analogue.
To conclude, we tested the proposed framework on a 3D real fluorescence microscopy dataset, showing good performance.

In future work we plan to consider more complicated 2D and 3D optical models as the ones in \cite{Denoyelle2019}, so as to verify the effectiveness of the approach in more realistic scenarios.

\backmatter

\bmhead{Acknowledgements}
The authors warmly thank the anonymous referees of this work who provided very detailed comments and thorough corrections which significantly improved the quality of the manuscript. 
We further thank R. Petit, B. Laville and L. Blanc-F\'eraud for useful discussions on both the theoretical and numerical aspects of this work. We thank J. Boulanger (MRC-LMB, Cambridge) for thoughtful discussions and for making this work possible by having put ML, LC and CE in touch with AM, who acquired the real 3D data presented in the Section \ref{sec:results}.
\bmhead{Funding}
This paper is part of MALIN, a project that has received funding from the European Research Council (ERC) under the European Union’s Horizon Europe  programme (grant agreement No. 101117133).
LC and ML acknowledge the support received by the projects ANR
MICROBLIND (ANR-21-CE48-0008) and 
ANR JCJC TASKABILE (ANR-22-CE48-0010). CE and ML acknowledge the support received by the MUR Excellence Department Project awarded to Dipartimento di Matematica, Università di Genova, CUP D33C23001110001. 
CE and ML acknowledge the support of the "PRIN 2022ANC8HL - Inverse Problems in the Imaging Sciences (IPIS)" project, granted by the Italian Ministero dell'Università e della Ricerca within the framework of the European Union - Next Generation EU program and of the Italian INdAM group on scientific calculus GNCS.

\bmhead{Data availability} 

The codes used for implementing the models and algorithms described in this work are available at \url{https://github.com/martalazzaretti/KL-TV-off-the-grid}.

\section*{Declarations}
\bmhead{Conflict of interest}
The authors declare that they have no conflict of interest.

\bibliography{main.bib}


\begin{thebibliography}{56}
\ifx \bisbn   \undefined \def \bisbn  #1{ISBN #1}\fi
\ifx \binits  \undefined \def \binits#1{#1}\fi
\ifx \bauthor  \undefined \def \bauthor#1{#1}\fi
\ifx \batitle  \undefined \def \batitle#1{#1}\fi
\ifx \bjtitle  \undefined \def \bjtitle#1{#1}\fi
\ifx \bvolume  \undefined \def \bvolume#1{\textbf{#1}}\fi
\ifx \byear  \undefined \def \byear#1{#1}\fi
\ifx \bissue  \undefined \def \bissue#1{#1}\fi
\ifx \bfpage  \undefined \def \bfpage#1{#1}\fi
\ifx \blpage  \undefined \def \blpage #1{#1}\fi
\ifx \burl  \undefined \def \burl#1{\textsf{#1}}\fi
\ifx \doiurl  \undefined \def \doiurl#1{\url{https://doi.org/#1}}\fi
\ifx \betal  \undefined \def \betal{\textit{et al.}}\fi
\ifx \binstitute  \undefined \def \binstitute#1{#1}\fi
\ifx \binstitutionaled  \undefined \def \binstitutionaled#1{#1}\fi
\ifx \bctitle  \undefined \def \bctitle#1{#1}\fi
\ifx \beditor  \undefined \def \beditor#1{#1}\fi
\ifx \bpublisher  \undefined \def \bpublisher#1{#1}\fi
\ifx \bbtitle  \undefined \def \bbtitle#1{#1}\fi
\ifx \bedition  \undefined \def \bedition#1{#1}\fi
\ifx \bseriesno  \undefined \def \bseriesno#1{#1}\fi
\ifx \blocation  \undefined \def \blocation#1{#1}\fi
\ifx \bsertitle  \undefined \def \bsertitle#1{#1}\fi
\ifx \bsnm \undefined \def \bsnm#1{#1}\fi
\ifx \bsuffix \undefined \def \bsuffix#1{#1}\fi
\ifx \bparticle \undefined \def \bparticle#1{#1}\fi
\ifx \barticle \undefined \def \barticle#1{#1}\fi
\bibcommenthead
\ifx \bconfdate \undefined \def \bconfdate #1{#1}\fi
\ifx \botherref \undefined \def \botherref #1{#1}\fi
\ifx \url \undefined \def \url#1{\textsf{#1}}\fi
\ifx \bchapter \undefined \def \bchapter#1{#1}\fi
\ifx \bbook \undefined \def \bbook#1{#1}\fi
\ifx \bcomment \undefined \def \bcomment#1{#1}\fi
\ifx \oauthor \undefined \def \oauthor#1{#1}\fi
\ifx \citeauthoryear \undefined \def \citeauthoryear#1{#1}\fi
\ifx \endbibitem  \undefined \def \endbibitem {}\fi
\ifx \bconflocation  \undefined \def \bconflocation#1{#1}\fi
\ifx \arxivurl  \undefined \def \arxivurl#1{\textsf{#1}}\fi
\csname PreBibitemsHook\endcsname

\bibitem[\protect\citeauthoryear{{Duval} and
  {Peyr{\'e}}}{2017a}]{Duval2017thingrids}
\begin{barticle}
\bauthor{\bsnm{{Duval}}, \binits{V.}},
\bauthor{\bsnm{{Peyr{\'e}}}, \binits{G.}}:
\batitle{{Sparse spikes super-resolution on thin grids II: the continuous basis
  pursuit}}.
\bjtitle{Inverse Probl.}
\bvolume{33}(\bissue{9}),
\bfpage{095008}
(\byear{2017})
\end{barticle}
\endbibitem

\bibitem[\protect\citeauthoryear{{Duval} and
  {Peyr{\'e}}}{2017b}]{Duval2017thingrids_b}
\begin{barticle}
\bauthor{\bsnm{{Duval}}, \binits{V.}},
\bauthor{\bsnm{{Peyr{\'e}}}, \binits{G.}}:
\batitle{{Sparse regularization on thin grids I: the Lasso}}.
\bjtitle{Inverse Probl.}
\bvolume{33}(\bissue{5}),
\bfpage{055008}
(\byear{2017})
\end{barticle}
\endbibitem

\bibitem[\protect\citeauthoryear{Laville et~al.}{2021}]{Laville2021}
\begin{botherref}
\oauthor{\bsnm{Laville}, \binits{B.}},
\oauthor{\bsnm{Blanc-Féraud}, \binits{L.}},
\oauthor{\bsnm{Aubert}, \binits{G.}}:
Off-the-grid variational sparse spike recovery: Methods and algorithms.
J. Imaging.
\textbf{7}(12)
(2021)
\end{botherref}
\endbibitem

\bibitem[\protect\citeauthoryear{Scherzer and Walch}{2009}]{Scherzer2009}
\begin{bchapter}
\bauthor{\bsnm{Scherzer}, \binits{O.}},
\bauthor{\bsnm{Walch}, \binits{B.}}:
\bctitle{Sparsity regularization for radon measures}.
In: \bbtitle{Proc.-SSVM 2009},
pp. \bfpage{452}--\blpage{463}.
\bpublisher{Springer},
\blocation{Berlin Heidelberg}
(\byear{2009})
\end{bchapter}
\endbibitem

\bibitem[\protect\citeauthoryear{Bredies and
  Pikkarainen}{2013}]{bredies2013inverse}
\begin{botherref}
\oauthor{\bsnm{Bredies}, \binits{K.}},
\oauthor{\bsnm{Pikkarainen}, \binits{H.K.}}:
Inverse problems in spaces of measures.
ESAIM: COCV.
(2013)
\end{botherref}
\endbibitem

\bibitem[\protect\citeauthoryear{De~Castro and
  Gamboa}{2012}]{decastro2012exact}
\begin{botherref}
\oauthor{\bsnm{De~Castro}, \binits{Y.}},
\oauthor{\bsnm{Gamboa}, \binits{F.}}:
{Exact reconstruction using Beurling minimal extrapolation}.
J. Math. Anal. Appl.
(2012)
\end{botherref}
\endbibitem

\bibitem[\protect\citeauthoryear{Fernandez-Granda}{2013}]{FernandezGranda2013}
\begin{bchapter}
\bauthor{\bsnm{Fernandez-Granda}, \binits{C.}}:
\bctitle{Support detection in super-resolution}.
In: \bbtitle{Proceedings of the 10th International Conference on Sampling
  Theory and Applications (SampTA 2013)},
pp. \bfpage{145}--\blpage{148}
(\byear{2013})
\end{bchapter}
\endbibitem

\bibitem[\protect\citeauthoryear{Duval and Peyr{\'e}}{2013}]{duval2014exact}
\begin{barticle}
\bauthor{\bsnm{Duval}, \binits{V.}},
\bauthor{\bsnm{Peyr{\'e}}, \binits{G.}}:
\batitle{Exact support recovery for sparse spikes deconvolution}.
\bjtitle{Found. Comput. Math.}
\bvolume{15},
\bfpage{1315}--\blpage{1355}
(\byear{2013})
\end{barticle}
\endbibitem

\bibitem[\protect\citeauthoryear{Denoyelle et~al.}{2017}]{denoyelle2017support}
\begin{botherref}
\oauthor{\bsnm{Denoyelle}, \binits{Q.}},
\oauthor{\bsnm{Duval}, \binits{V.}},
\oauthor{\bsnm{Peyré}, \binits{G.}}:
Support recovery for sparse super-resolution of positive measures.
J. Fourier Anal. Appl.
(2017)
\end{botherref}
\endbibitem

\bibitem[\protect\citeauthoryear{Boyer et~al.}{2017}]{boyer2017adapting}
\begin{botherref}
\oauthor{\bsnm{Boyer}, \binits{C.}},
\oauthor{\bsnm{De~Castro}, \binits{Y.}},
\oauthor{\bsnm{Salmon}, \binits{J.}}:
Adapting to unknown noise level in sparse deconvolution.
Inf. Inference J. IMA
(2017)
\end{botherref}
\endbibitem

\bibitem[\protect\citeauthoryear{Poon and Peyré}{2019}]{poon2019multi}
\begin{botherref}
\oauthor{\bsnm{Poon}, \binits{C.}},
\oauthor{\bsnm{Peyré}, \binits{G.}}:
Multi-dimensional sparse super-resolution.
SIAM J. Math. Anal.
(2019)
\end{botherref}
\endbibitem

\bibitem[\protect\citeauthoryear{Denoyelle et~al.}{2019}]{Denoyelle2019}
\begin{barticle}
\bauthor{\bsnm{Denoyelle}, \binits{Q.}},
\bauthor{\bsnm{Duval}, \binits{V.}},
\bauthor{\bsnm{Peyré}, \binits{G.}},
\bauthor{\bsnm{Soubies}, \binits{E.}}:
\batitle{{The sliding Frank–Wolfe algorithm and its application to
  super-resolution microscopy}}.
\bjtitle{Inverse Probl.}
\bvolume{36}(\bissue{1}),
\bfpage{014001}
(\byear{2019})
\end{barticle}
\endbibitem

\bibitem[\protect\citeauthoryear{Flinth et~al.}{2021}]{Flinth}
\begin{barticle}
\bauthor{\bsnm{Flinth}, \binits{A.}},
\bauthor{\bsnm{Gournay}, \binits{F.}},
\bauthor{\bsnm{Weiss}, \binits{P.}}:
\batitle{On the linear convergence rates of exchange and continuous methods for
  total variation minimization}.
\bjtitle{Mathematical Programming}
\bvolume{190}(\bissue{1}),
\bfpage{221}--\blpage{257}
(\byear{2021})
\end{barticle}
\endbibitem

\bibitem[\protect\citeauthoryear{Bénard et~al.}{2024}]{Traonmilin2024}
\begin{botherref}
\oauthor{\bsnm{Bénard}, \binits{P.-J.}},
\oauthor{\bsnm{Traonmilin}, \binits{Y.}},
\oauthor{\bsnm{Aujol}, \binits{J.-F.}},
\oauthor{\bsnm{Soubies}, \binits{E.}}:
Estimation of off-the grid sparse spikes with over-parametrized projected
  gradient descent: theory and application.
Inverse Problems
(2024)
\end{botherref}
\endbibitem

\bibitem[\protect\citeauthoryear{Dossal et~al.}{2017}]{Dossal2017}
\begin{barticle}
\bauthor{\bsnm{Dossal}, \binits{C.}},
\bauthor{\bsnm{Duval}, \binits{V.}},
\bauthor{\bsnm{Poon}, \binits{C.}}:
\batitle{Sampling the fourier transform along radial lines}.
\bjtitle{SIAM J. Numer. Anal.}
\bvolume{55}(\bissue{6}),
\bfpage{2540}--\blpage{2564}
(\byear{2017})
\end{barticle}
\endbibitem

\bibitem[\protect\citeauthoryear{De~Castro et~al.}{2021}]{DeCastro2021}
\begin{barticle}
\bauthor{\bsnm{De~Castro}, \binits{Y.}},
\bauthor{\bsnm{Gadat}, \binits{S.}},
\bauthor{\bsnm{C.}, \binits{M.}},
\bauthor{\bsnm{Maugis-Rabusseau}, \binits{C.}}:
\batitle{{SuperMix: Sparse regularization for mixtures}}.
\bjtitle{Ann. Statist.}
\bvolume{49}(\bissue{3}),
\bfpage{1779}--\blpage{1809}
(\byear{2021})
\end{barticle}
\endbibitem

\bibitem[\protect\citeauthoryear{Pouchol and Verdier}{2024}]{Pouchol2024}
\begin{barticle}
\bauthor{\bsnm{Pouchol}, \binits{C.}},
\bauthor{\bsnm{Verdier}, \binits{O.}}:
\batitle{Linear inverse problems with nonnegativity constraints: Singularity of
  optimisers}.
\bjtitle{Inverse Problems and Imaging}
\bvolume{18}(\bissue{1}),
\bfpage{138}--\blpage{164}
(\byear{2024})
\end{barticle}
\endbibitem

\bibitem[\protect\citeauthoryear{Bertero et~al.}{2018}]{Bertero2018}
\begin{bbook}
\bauthor{\bsnm{Bertero}, \binits{M.}},
\bauthor{\bsnm{Boccacci}, \binits{P.}},
\bauthor{\bsnm{Ruggiero}, \binits{V.}}:
\bbtitle{Inverse Imaging with Poisson Data}.
\bsertitle{2053-2563}.
\bpublisher{IOP Publishing},
\blocation{UK}
(\byear{2018})
\end{bbook}
\endbibitem

\bibitem[\protect\citeauthoryear{Poon}{2019}]{poon2019sparse}
\begin{botherref}
\oauthor{\bsnm{Poon}, \binits{C.}}:
An Introduction to Sparse Spikes Recovery via the BLASSO.
Lecture notes
(2019)
\end{botherref}
\endbibitem

\bibitem[\protect\citeauthoryear{Rudin}{1987}]{Rudin1987}
\begin{bbook}
\bauthor{\bsnm{Rudin}, \binits{W.}}:
\bbtitle{Real and Complex Analysis, 3rd Ed.}
\bpublisher{McGraw-Hill, Inc.},
\blocation{USA}
(\byear{1987})
\end{bbook}
\endbibitem

\bibitem[\protect\citeauthoryear{Cohn}{2013}]{cohn2013measure}
\begin{bbook}
\bauthor{\bsnm{Cohn}, \binits{D.L.}}:
\bbtitle{Measure Theory}.
\bpublisher{Springer},
\blocation{New York}
(\byear{2013})
\end{bbook}
\endbibitem

\bibitem[\protect\citeauthoryear{Denoyelle}{2018}]{denoyelle2018_phdthesis}
\begin{botherref}
\oauthor{\bsnm{Denoyelle}, \binits{Q.}}:
{Theoretical and Numerical Analysis of Super-Resolution Without Grid}.
Theses,
{Universit{\'e} Paris sciences et lettres}
(July 2018)
\end{botherref}
\endbibitem

\bibitem[\protect\citeauthoryear{Javanshiri and
  Nasr-Isfahani}{2013}]{JAVANSHIRI2013}
\begin{barticle}
\bauthor{\bsnm{Javanshiri}, \binits{H.}},
\bauthor{\bsnm{Nasr-Isfahani}, \binits{R.}}:
\batitle{{The strict topology for the space of Radon measures on a locally
  compact Hausdorff space}}.
\bjtitle{Topology and its Applications}
\bvolume{160}(\bissue{7}),
\bfpage{887}--\blpage{895}
(\byear{2013})
\end{barticle}
\endbibitem

\bibitem[\protect\citeauthoryear{Chizat and Bach}{2018}]{Chizat2018OnTG}
\begin{bchapter}
\bauthor{\bsnm{Chizat}, \binits{L.}},
\bauthor{\bsnm{Bach}, \binits{F.}}:
\bctitle{On the global convergence of gradient descent for over-parameterized
  models using optimal transport}.
In: \bbtitle{Proc. 32nd Int. Conf. Neural Inf. Process. Syst., NIPS'18},
pp. \bfpage{3040}--\blpage{3050}
(\byear{2018})
\end{bchapter}
\endbibitem

\bibitem[\protect\citeauthoryear{Koulouri
  et~al.}{2021}]{Koulouri_2021_AdaptiveSuperRes}
\begin{barticle}
\bauthor{\bsnm{Koulouri}, \binits{A.}},
\bauthor{\bsnm{Heins}, \binits{P.}},
\bauthor{\bsnm{Burger}, \binits{M.}}:
\batitle{Adaptive superresolution in deconvolution of sparse peaks}.
\bjtitle{IEEE Trans. Signal Process.}
\bvolume{69},
\bfpage{165}--\blpage{178}
(\byear{2021})
\end{barticle}
\endbibitem

\bibitem[\protect\citeauthoryear{Laville et~al.}{2023a}]{LavilleSIAMIS2023}
\begin{barticle}
\bauthor{\bsnm{Laville}, \binits{B.}},
\bauthor{\bsnm{Blanc-F\'{e}raud}, \binits{L.}},
\bauthor{\bsnm{Aubert}, \binits{G.}}:
\batitle{Off-the-grid curve reconstruction through divergence regularization:
  An extreme point result}.
\bjtitle{SIAM J. Imaging Sci.}
\bvolume{16}(\bissue{2}),
\bfpage{867}--\blpage{885}
(\byear{2023})
\end{barticle}
\endbibitem

\bibitem[\protect\citeauthoryear{Laville et~al.}{2023b}]{LavilleSSVM2023}
\begin{bchapter}
\bauthor{\bsnm{Laville}, \binits{B.}},
\bauthor{\bsnm{Blanc-F{\'e}raud}, \binits{L.}},
\bauthor{\bsnm{Aubert}, \binits{G.}}:
\bctitle{Off-the-grid charge algorithm for curve reconstruction in inverse
  problems}.
In: \bbtitle{Proc.-SSVM 2023},
pp. \bfpage{393}--\blpage{405}.
\bpublisher{Springer},
\blocation{Cham}
(\byear{2023})
\end{bchapter}
\endbibitem

\bibitem[\protect\citeauthoryear{De~Castro et~al.}{2023}]{DeCastro2023}
\begin{barticle}
\bauthor{\bsnm{De~Castro}, \binits{Y.}},
\bauthor{\bsnm{Duval}, \binits{V.}},
\bauthor{\bsnm{Petit}, \binits{R.}}:
\batitle{Towards off-the-grid algorithms for total variation regularized
  inverse problems}.
\bjtitle{J. Math. Imaging Vis.}
\bvolume{65}(\bissue{1}),
\bfpage{53}--\blpage{81}
(\byear{2023})
\end{barticle}
\endbibitem

\bibitem[\protect\citeauthoryear{De~Castro et~al.}{2021}]{DeCastroPetit2021}
\begin{bchapter}
\bauthor{\bsnm{De~Castro}, \binits{Y.}},
\bauthor{\bsnm{Duval}, \binits{V.}},
\bauthor{\bsnm{Petit}, \binits{R.}}:
\bctitle{Towards off-the-grid algorithms for total variation regularized
  inverse problems}.
In: \bbtitle{Proceedings SSVM 2021},
pp. \bfpage{553}--\blpage{564}.
\bpublisher{Springer},
\blocation{Cham}
(\byear{2021})
\end{bchapter}
\endbibitem

\bibitem[\protect\citeauthoryear{De~Castro et~al.}{2024}]{DeCastro2024exact}
\begin{barticle}
\bauthor{\bsnm{De~Castro}, \binits{Y.}},
\bauthor{\bsnm{Duval}, \binits{V.}},
\bauthor{\bsnm{Petit}, \binits{R.}}:
\batitle{Exact recovery of the support of piecewise constant images via total
  variation regularization}.
\bjtitle{Inverse Problems}
\bvolume{40}(\bissue{10}),
\bfpage{105012}
(\byear{2024})
\end{barticle}
\endbibitem

\bibitem[\protect\citeauthoryear{Beurling}{1938}]{beurling1938}
\begin{bchapter}
\bauthor{\bsnm{Beurling}, \binits{A.}}:
\bctitle{Sur les intégrales de fourier absolument convergentes et leur
  application à une transformation fonctionnelle}.
In: \bbtitle{Proc. Ninth Scand. Math. Congr.},
\bconflocation{Helsinki, Finland},
pp. \bfpage{345}--\blpage{366}
(\byear{1938})
\end{bchapter}
\endbibitem

\bibitem[\protect\citeauthoryear{Bertero et~al.}{2009}]{Bertero2009ImageDW}
\begin{barticle}
\bauthor{\bsnm{Bertero}, \binits{M.}},
\bauthor{\bsnm{Boccacci}, \binits{P.}},
\bauthor{\bsnm{Desidera}, \binits{G.}},
\bauthor{\bsnm{Vicidomini}, \binits{G.}}:
\batitle{{Image deblurring with Poisson data: from cells to galaxies}}.
\bjtitle{Inverse Probl.}
\bvolume{25},
\bfpage{123006}
(\byear{2009})
\end{barticle}
\endbibitem

\bibitem[\protect\citeauthoryear{Bertero et~al.}{2010}]{Bertero2010ADP}
\begin{barticle}
\bauthor{\bsnm{Bertero}, \binits{M.}},
\bauthor{\bsnm{Boccacci}, \binits{P.}},
\bauthor{\bsnm{Talenti}, \binits{G.}},
\bauthor{\bsnm{Zanella}, \binits{R.}},
\bauthor{\bsnm{Zanni}, \binits{L.}}:
\batitle{{A discrepancy principle for Poisson data}}.
\bjtitle{Inverse Probl.}
\bvolume{26},
\bfpage{105004}
(\byear{2010})
\end{barticle}
\endbibitem

\bibitem[\protect\citeauthoryear{Li et~al.}{2015}]{LiShen2015}
\begin{barticle}
\bauthor{\bsnm{Li}, \binits{J.}},
\bauthor{\bsnm{Shen}, \binits{Z.}},
\bauthor{\bsnm{Yin}, \binits{R.}},
\bauthor{\bsnm{Zhang}, \binits{X.}}:
\batitle{{A reweighted $l^2$ method for image restoration with Poisson and
  mixed Poisson-Gaussian noise}}.
\bjtitle{Inverse Probl. Imaging.}
\bvolume{9}(\bissue{3}),
\bfpage{875}--\blpage{894}
(\byear{2015})
\end{barticle}
\endbibitem

\bibitem[\protect\citeauthoryear{Sawatzky et~al.}{2013}]{Sawatzky2013}
\begin{bbook}
\bauthor{\bsnm{Sawatzky}, \binits{A.}},
\bauthor{\bsnm{Brune}, \binits{C.}},
\bauthor{\bsnm{K{\"o}sters}, \binits{T.}},
\bauthor{\bsnm{W{\"u}bbeling}, \binits{F.}},
\bauthor{\bsnm{Burger}, \binits{M.}}:
\bbtitle{EM-TV Methods for Inverse Problems with Poisson Noise},
pp. \bfpage{71}--\blpage{142}.
\bpublisher{Springer},
\blocation{Cham}
(\byear{2013})
\end{bbook}
\endbibitem

\bibitem[\protect\citeauthoryear{Lanza et~al.}{2014}]{Lanza2014}
\begin{barticle}
\bauthor{\bsnm{Lanza}, \binits{A.}},
\bauthor{\bsnm{Morigi}, \binits{S.}},
\bauthor{\bsnm{Sgallari}, \binits{F.}},
\bauthor{\bsnm{Wen}, \binits{Y.-W.}}:
\batitle{{Image restoration with Poisson–Gaussian mixed noise}}.
\bjtitle{"Comput. Methods Biomech. Biomed. Eng. Imaging Vis.}
\bvolume{2}(\bissue{1}),
\bfpage{12}--\blpage{24}
(\byear{2014})
\end{barticle}
\endbibitem

\bibitem[\protect\citeauthoryear{Rockafellar}{1970}]{rockafellar1970convex}
\begin{bbook}
\bauthor{\bsnm{Rockafellar}, \binits{R.T.}}:
\bbtitle{Convex Analysis}.
\bpublisher{Princeton University Press},
\blocation{Princeton, NJ, USA}
(\byear{1970})
\end{bbook}
\endbibitem

\bibitem[\protect\citeauthoryear{Ekeland and Témam}{1999}]{Ekeland1999}
\begin{bbook}
\bauthor{\bsnm{Ekeland}, \binits{I.}},
\bauthor{\bsnm{Témam}, \binits{R.}}:
\bbtitle{Convex Analysis and Variational Problems}.
\bpublisher{Society for Industrial and Applied Mathematics},
\blocation{Philadelphia, PA, USA}
(\byear{1999})
\end{bbook}
\endbibitem

\bibitem[\protect\citeauthoryear{Combettes and Pesquet}{2011}]{Combettes2011}
\begin{bbook}
\bauthor{\bsnm{Combettes}, \binits{P.L.}},
\bauthor{\bsnm{Pesquet}, \binits{J.-C.}}:
\bbtitle{Proximal Splitting Methods in Signal Processing},
pp. \bfpage{185}--\blpage{212}.
\bpublisher{Springer},
\blocation{New York, NY}
(\byear{2011})
\end{bbook}
\endbibitem

\bibitem[\protect\citeauthoryear{Ndiaye et~al.}{2017}]{Ndiaye2017}
\begin{barticle}
\bauthor{\bsnm{Ndiaye}, \binits{E.}},
\bauthor{\bsnm{Fercoq}, \binits{O.}},
\bauthor{\bsnm{Gramfort}, \binits{A.}},
\bauthor{\bsnm{Salmon}, \binits{J.}}:
\batitle{Gap safe screening rules for sparsity enforcing penalties}.
\bjtitle{J. Mach. Learn. Res.}
\bvolume{18}(\bissue{1}),
\bfpage{4671}--\blpage{4703}
(\byear{2017})
\end{barticle}
\endbibitem

\bibitem[\protect\citeauthoryear{Wang and Ye}{2014}]{Wang2014}
\begin{bchapter}
\bauthor{\bsnm{Wang}, \binits{J.}},
\bauthor{\bsnm{Ye}, \binits{J.}}:
\bctitle{{Two-Layer Feature Reduction for Sparse-Group LASSO via Decomposition
  of Convex Sets}}.
In: \bbtitle{Proc. 27th Int. Conf. Neural Inf. Process. Syst. - Vol. 2,
  NIPS'14},
pp. \bfpage{2132}--\blpage{2140}.
\bpublisher{MIT Press},
\blocation{Cambridge, MA, USA}
(\byear{2014})
\end{bchapter}
\endbibitem

\bibitem[\protect\citeauthoryear{Dantas et~al.}{2021}]{Dantas2021}
\begin{bchapter}
\bauthor{\bsnm{Dantas}, \binits{C.F.}},
\bauthor{\bsnm{Soubies}, \binits{E.}},
\bauthor{\bsnm{F{\'e}votte}, \binits{C.}}:
\bctitle{{Safe screening for sparse regression with the Kullback-Leibler
  divergence}}.
In: \bbtitle{Proc. - ICASSP IEEE Int. Conf. Acoust.},
\bconflocation{Toronto, Canada}
(\byear{2021})
\end{bchapter}
\endbibitem

\bibitem[\protect\citeauthoryear{Ekeland and
  T{\'e}mam}{1999}]{ekeland1999convex}
\begin{bbook}
\bauthor{\bsnm{Ekeland}, \binits{I.}},
\bauthor{\bsnm{T{\'e}mam}, \binits{R.}}:
\bbtitle{Convex Analysis and Variational Problems}.
\bpublisher{Society for Industrial and Applied Mathematics},
\blocation{Philadelphia, PA, USA}
(\byear{1999})
\end{bbook}
\endbibitem

\bibitem[\protect\citeauthoryear{Valkonen}{2023}]{Valkonen2023}
\begin{botherref}
\oauthor{\bsnm{Valkonen}, \binits{T.}}:
{Proximal methods for point source localisation}.
{J. Nonlinear Anal. Optim.}
\textbf{{Volume 4}}
(2023)
\end{botherref}
\endbibitem

\bibitem[\protect\citeauthoryear{Frank and Wolfe}{1956}]{Frank1956}
\begin{barticle}
\bauthor{\bsnm{Frank}, \binits{M.}},
\bauthor{\bsnm{Wolfe}, \binits{P.}}:
\batitle{An algorithm for quadratic programming}.
\bjtitle{Nav. Res. Logist. Q.}
\bvolume{3}(\bissue{1-2}),
\bfpage{95}--\blpage{110}
(\byear{1956})
\end{barticle}
\endbibitem

\bibitem[\protect\citeauthoryear{Bredies et~al.}{2009}]{BrediesLorenz2009}
\begin{barticle}
\bauthor{\bsnm{Bredies}, \binits{K.}},
\bauthor{\bsnm{Lorenz}, \binits{D.A.}},
\bauthor{\bsnm{Maass}, \binits{P.}}:
\batitle{A generalized conditional gradient method and its connection to an
  iterative shrinkage method}.
\bjtitle{Computational Optimization and Applications}
\bvolume{42}(\bissue{2}),
\bfpage{173}--\blpage{193}
(\byear{2009})
\end{barticle}
\endbibitem

\bibitem[\protect\citeauthoryear{Chizat}{2022}]{Chizat2021}
\begin{barticle}
\bauthor{\bsnm{Chizat}, \binits{L.}}:
\batitle{Sparse optimization on measures with over-parameterized gradient
  descent}.
\bjtitle{Math. Program.}
\bvolume{194}(\bissue{1–2}),
\bfpage{487}--\blpage{532}
(\byear{2022})
\end{barticle}
\endbibitem

\bibitem[\protect\citeauthoryear{Pokutta}{2024}]{Pokutta}
\begin{barticle}
\bauthor{\bsnm{Pokutta}, \binits{S.}}:
\batitle{{The Frank-Wolfe Algorithm: A Short Introduction}}.
\bjtitle{Jahresbericht der Deutschen Mathematiker-Vereinigung}
\bvolume{126}(\bissue{1}),
\bfpage{3}--\blpage{35}
(\byear{2024})
\end{barticle}
\endbibitem

\bibitem[\protect\citeauthoryear{Courbot and Colicchio}{2021}]{Courbot2021}
\begin{barticle}
\bauthor{\bsnm{Courbot}, \binits{J.-B.}},
\bauthor{\bsnm{Colicchio}, \binits{B.}}:
\batitle{A fast homotopy algorithm for gridless sparse recovery}.
\bjtitle{Inverse Probl.}
\bvolume{37}(\bissue{2}),
\bfpage{025002}
(\byear{2021})
\end{barticle}
\endbibitem

\bibitem[\protect\citeauthoryear{Osborne et~al.}{2000a}]{Osborne2000}
\begin{barticle}
\bauthor{\bsnm{Osborne}, \binits{M.}},
\bauthor{\bsnm{Presnell}, \binits{B.}},
\bauthor{\bsnm{Turlach}, \binits{B.}}:
\batitle{{A new approach to variable selection in least squares problems}}.
\bjtitle{IMA J. Numer. Anal.}
\bvolume{20}(\bissue{3}),
\bfpage{389}--\blpage{403}
(\byear{2000})
\end{barticle}
\endbibitem

\bibitem[\protect\citeauthoryear{Osborne et~al.}{2000b}]{Osborne2000b}
\begin{barticle}
\bauthor{\bsnm{Osborne}, \binits{M.R.}},
\bauthor{\bsnm{Presnell}, \binits{B.}},
\bauthor{\bsnm{Turlach}, \binits{B.A.}}:
\batitle{{On the LASSO and Its Dual}}.
\bjtitle{J. Comput. Graph. Stat.}
\bvolume{9}(\bissue{2}),
\bfpage{319}--\blpage{337}
(\byear{2000})
\end{barticle}
\endbibitem

\bibitem[\protect\citeauthoryear{Harmany et~al.}{2012}]{Harmany2012}
\begin{barticle}
\bauthor{\bsnm{Harmany}, \binits{Z.T.}},
\bauthor{\bsnm{Marcia}, \binits{R.F.}},
\bauthor{\bsnm{Willett}, \binits{R.M.}}:
\batitle{{This is SPIRAL-TAP: Sparse Poisson Intensity Reconstruction
  ALgorithms—Theory and Practice}}.
\bjtitle{IEEE Trans. Image Process.}
\bvolume{21}(\bissue{3}),
\bfpage{1084}--\blpage{1096}
(\byear{2012})
\end{barticle}
\endbibitem

\bibitem[\protect\citeauthoryear{Calatroni et~al.}{2017}]{Calatroni2017}
\begin{barticle}
\bauthor{\bsnm{Calatroni}, \binits{L.}},
\bauthor{\bsnm{De~Los~Reyes}, \binits{J.C.}},
\bauthor{\bsnm{Sch\"{o}nlieb}, \binits{C.-B.}}:
\batitle{Infimal convolution of data discrepancies for mixed noise removal}.
\bjtitle{SIAM J. Imaging Sci.}
\bvolume{10}(\bissue{3}),
\bfpage{1196}--\blpage{1233}
(\byear{2017})
\end{barticle}
\endbibitem

\bibitem[\protect\citeauthoryear{Johnson et~al.}{2024}]{Johnson24phase}
\begin{barticle}
\bauthor{\bsnm{Johnson}, \binits{C.}},
\bauthor{\bsnm{Guo}, \binits{M.}},
\bauthor{\bsnm{Schneider}, \binits{M.C.}},
\bauthor{\bsnm{Su}, \binits{Y.}},
\bauthor{\bsnm{Khuon}, \binits{S.}},
\bauthor{\bsnm{Reiser}, \binits{N.}},
\bauthor{\bsnm{Wu}, \binits{Y.}},
\bauthor{\bsnm{Riviere}, \binits{P.L.}},
\bauthor{\bsnm{Shroff}, \binits{H.}}:
\batitle{Phase-diversity-based wavefront sensing for fluorescence microscopy}.
\bjtitle{Optica}
\bvolume{11}(\bissue{6}),
\bfpage{806}--\blpage{820}
(\byear{2024})
\end{barticle}
\endbibitem

\bibitem[\protect\citeauthoryear{Bevilacqua
  et~al.}{2023}]{Bevilacqua_2023_masked}
\begin{barticle}
\bauthor{\bsnm{Bevilacqua}, \binits{F.}},
\bauthor{\bsnm{Lanza}, \binits{A.}},
\bauthor{\bsnm{Pragliola}, \binits{M.}},
\bauthor{\bsnm{Sgallari}, \binits{F.}}:
\batitle{Masked unbiased principles for parameter selection in variational
  image restoration under poisson noise}.
\bjtitle{Inverse Probl.}
\bvolume{39}(\bissue{3}),
\bfpage{034002}
(\byear{2023})
\end{barticle}
\endbibitem

\bibitem[\protect\citeauthoryear{Gomez-Navarro
  et~al.}{2020}]{Navarro2020_eres3d}
\begin{barticle}
\bauthor{\bsnm{Gomez-Navarro}, \binits{N.}},
\bauthor{\bsnm{Melero}, \binits{A.}},
\bauthor{\bsnm{Li}, \binits{X.-H.}},
\bauthor{\bsnm{Boulanger}, \binits{J.}},
\bauthor{\bsnm{Kukulski}, \binits{W.}},
\bauthor{\bsnm{Miller}, \binits{E.A.}}:
\batitle{{Cargo crowding contributes to sorting stringency in COPII vesicles}}.
\bjtitle{J. Cell. Biol.}
\bvolume{219}(\bissue{7}),
\bfpage{201806038}
(\byear{2020})
\end{barticle}
\endbibitem

\end{thebibliography}

\end{document}